\documentclass[10pt]{article}

\usepackage{amsmath}   
\usepackage{amssymb}
\usepackage{amsthm}
\usepackage{yhmath} %needed widehats
\usepackage{tikz-cd} %needed for diagrams
\usepackage{accents} % needed for the dot t-derivative

\newtheorem*{Th*}{Theorem}
\newtheorem{Th}{Theorem}[section]
\newtheorem{Prop}{Proposition}[section]   
\newtheorem{Lem}{Lemma}[section]   
\newtheorem{Coro}{Corollary}[section]   
\newtheorem{Def}{Definition}  
\newtheorem{Rem}{Remark}[section]

\newcommand{\R}{\mathbb{R}}
\newcommand{\Z}{\mathbb{Z}}
\newcommand{\C}{\mathbb{C}}
\newcommand{\T}{\mathbb{T}}

\newcommand{\PP}{\mathcal{P}}

\newcommand{\Om}{\Omega}
\newcommand{\om}{\omega}

\newcommand{\id}{{\rm id}}
\newcommand{\Id}{{\rm Id}}
\newcommand{\p}{{\partial}}

\newcommand{\rk}{\mathop{\rm rk}}
\newcommand{\m}{\langle m\rangle}
\newcommand{\Div}{\mathop{\rm div}}
\newcommand{\pp}{{\rm P}}
\newcommand{\tr}{\mathop{\rm tr}}
\newcommand{\dt}[1]{\accentset{\mbox{\bfseries .}}{#1}}

\begin{document}

\title{Spatially quasi-periodic solutions of the Euler equation}   
 
\author{Xu Sun, P. Topalov\thanks{P.T. is partially supported by the Simons Foundation, 
Award \#526907}} 
  
\maketitle

\begin{abstract}  
We develop a framework for studying quasi-periodic maps and diffeomorphisms on $\R^n$.
As an application, we prove that the Euler equation is locally well posed in a space of quasi-periodic vector fields on $\R^n$.
In particular, the equation preserves the spatial quasi-periodicity of the initial data. Several results on the analytic dependence of solutions 
on the time and the initial data are proved.
\end{abstract}

\section{Introduction}\label{sec:introduction}
Consider the Euler equation in $\R^n$ ($n\ge 2$),
\begin{equation}\label{eq:euler}
\left\{
\begin{array}{l}
u_t+u\cdot\nabla u=-\nabla\pp,\quad\Div u=0,\\
u|_{t=0}=u_0,
\end{array}
\right.
\end{equation}
where $u$ is the fluid velocity, $\pp$ is the (scalar) pressure, and $\nabla$ denotes the gradient
in the Euclidean space taken component-wise. In this paper we will prove that the Euler equation is well posed
in a class of quasi-periodic vector fields on $\R^n$. Quasi-periodic functions appear naturally in integrable PDE's where the solutions 
can be frequently written in terms of theta functions (see e.g. \cite{KT,Moser} and the references therein). Such solutions are
quasi-periodic both in time and spatial direction. Note that in oceanography, modulational instabilities of periodic wave trains introduce perturbations 
that lead to spatially quasi-periodic dynamics and are believed to be one of the mechanisms responsible for the formation of 
rogue waves (\cite{OOS,WZ,AH}).  Faraday wave experiment in which a fluid layer is subject to vertical oscillations leads to quasipatterns
which are quasi-periodic functions. For example, a twelve-fold orientationally symmetric quasipattern is produced by forcing a layer of
silicone oil simultaneously at two frequencies (\cite{EF}). Quasipatterns were found in nonlinear optical systems, shaken convection,
and in liquid crystals (see \cite{Yu1,Yu2,Iooss1,Iooss2,Iooss3} and the references therein).

In order to define the quasi-periodic functions on $\R^n$ with $n\ge 1$ we fix an integer $M\ge n$, and a linear map
\begin{equation}\label{eq:Om}
\Omega : \R^n\to\R^M
\end{equation}
of rank $n$, $\rk\Omega=n$. Consider the standard covering map ${\tt p} : \R^M\to\T^M$ where 
$\T^M:=\R^M/\Z^M$ is the $M$-dimensional torus. Denote by $\Om_{\tt p}$ the composed map 
${\tt p}\circ\Om : \R^n\to\T^M$ and let $N\ge 1$ be an integer.

\begin{Def}\label{def:qp}
A map $f : \R^n\to\R^N$ is called {\em quasi-periodic} if $f(x)=F\big(\Om_{\tt p}(x)\big)$ where $F: \T^M\to\R^N$ is a continuous map.
\end{Def}

\noindent In the case when $n=1$ the definition above coincides with the classical definition of quasi-periodic functions 
(cf. \cite{Bohl,Esclangon}).
We will assume that $\Om$ satisfies the following {\em non-resonance condition}:
\begin{itemize}
\item[{\rm\bf (NC)}] {\em The image of the map $\Om_{\rm p} : \R^n\to\T^M$ is dense in $\T^M$.}
\end{itemize}

\begin{Rem}\label{rem:NC}
By Lemma \ref{lem:density<->injectivity} in the Appendix the non-resonance condition {\rm (NC)} is equivalent to 
the injectivity of the map 
\begin{equation}\label{eq:Lambda_m-map}
\Z^M\to\R^n,\quad m\mapsto\Lambda_m,
\end{equation}
where $\Lambda_m:=2\pi\Om^T(m)$ and $(\cdot)^T$ denotes the transpose of a matrix. 
By the Kronecker-Weyl theorem, {\rm (NC)} holds if e.g. there exists $\alpha\in\R^n$ such that 
the components of $\Om(\alpha)\in\R^M$ are linearly independent over $\Z$, i.e., 
if $\big(m,\Om(\alpha)\big)=0$ implies $m=0$.
\end{Rem}

\noindent Here and below $(\cdot,\cdot)$ denotes the Euclidean scalar product. By $\Z_{\ge 0}$ we will denote the set of 
non-negative integer numbers.

\medskip

\noindent{\em The space $Q^s_\Om(\R^n)$:}
We will concentrate our attention to the case when $N=1$. The case $N>1$ will then easily follow.
For any $s>\frac{M}{2}$ consider the Sobolev space $H^s(\T^M)\equiv H^s(\T^M,\R)$ of maps $\T^M\to\R$. 
By the Sobolev embedding theorem, the space $H^s(\T^M)$ is compactly embedded 
in the space of continuous functions on the torus $C(\T^M)\equiv C(\T^M,\R)$.
For a given $s>\frac{M}{2}$ and $\Om : \R^n\to\R^M$, as above, define the following space of quasi-periodic functions
\begin{equation}\label{eq:Q^s}
Q^s_\Om(\R^n):=\big\{f(x)=F\big(\Om_{\tt p}(x)\big)\,\big|\,F\in H^s(\T^M)\big\}\,.
\end{equation}
Take $f\in Q^s_\Om(\R^n)$.
Since $s>\frac{M}{2}$, the function $F\in H^s(\T^M)$ has a uniformly convergent Fourier series. 
This implies that $f(x)=F(\Om_{\tt p}(x))$ has a uniformly convergent expansion
\begin{eqnarray}
f(x)&=&\sum_{m\in \Z^M}\widehat{F}_m e^{2\pi i (m,\Om(x))}\nonumber\\
&=&\sum_{m\in \Z^M}\widehat{F}_m e^{i (\Lambda_m,x)},\quad\quad
\Lambda_m\equiv 2\pi\Om^T(m)\in\R^n,\label{eq:f_expansion}
\end{eqnarray}
where $\widehat{F}_m$, $m\in\Z^M$, are the Fourier coefficients of the function $F\in H^s(\T^M)$. 
The uniform convergence in \eqref{eq:f_expansion} implies that $f$ is a uniform limit of 
trigonometric polynomials, and hence by Wiener characterization (see e.g. \cite[Ch. I, \S 7]{Levitan},\cite[\S 2]{SunTopalov}) it 
represents a (classical) {\em almost-periodic} function in the sense of Bohr with Fourier exponents 
$\Lambda_m\equiv2\pi\Om^T(m)$. Hence,
\[
Q^s_\Om(\R^n)\subseteq C_{ap}(\R^n)\subseteq C_b(\R^n)
\]
where $C_{ap}(\R^n)$ is the space of (classical) almost-periodic functions in $\R^n$ in the sense of Bohr and 
$C_b(\R^n)$ is the space of uniformly bounded continuous functions in $\R^n$ (for more details see the short discussion at 
the end of the Appendix). It follows from Remark \ref{rem:NC} that the map \eqref{eq:Lambda_m-map} is injective. 
In particular, we see from the uniform convergence of \eqref{eq:f_expansion} that for any $m\in\Z^M$,
\begin{equation}\label{eq:F_m}
\widehat{F}_m=\lim_{T\to\infty}\frac{1}{(2T)^n}\int_{[-T,T]^n}f(x)\,e^{- i(\Lambda_m,x)}\,dx,
\end{equation}
where the integration is over the cube $[-T,T]^n$ in $\R^n$. We define the norm in $Q^s_\Om(\R^n)$, $s>M/2$,
\begin{equation}\label{eq:Q^s-norm}
\|f\|_s:=\Big(\sum_{m\in\Z^M}|\widehat{F}_m|^2\m^{2s}\Big)^{1/2},\quad\m:=\sqrt{1+|m|^2}\,.
\end{equation}

For $F\in H^s(\T^M)$ consider the pull-back map $\Om_{\tt p}^*(F)(x):=F\big(\Om_{\tt p}(x)\big)\in Q^s_\Om(\R^n)$. 
By construction, we have

\begin{Lem}\label{lem:pull-back}
Assume that $s>M/2$. Then, the map $H^s(\T^M)\stackrel{\Om_{\tt p}^*}{\to} Q^s_\Om(\R^n)$, $F\mapsto\Om_{\tt p}^*(F)$,
is a linear isomorphism.
\end{Lem}
\noindent In particular, we see that $Q^s_\Om(\R^n)$ is a Hilbert space with a scalar product
\begin{equation}\label{eq:Q^s-scalar_product}
(f,g)_s:=\sum_{m\in\Z^M}\widehat{F}_m\overline{\big(\widehat{G}_m\big)}\m^{2s},\quad f,g\in Q^s_\Om(\R^n).
\end{equation}
In addition to the scalar product \eqref{eq:Q^s-scalar_product} on $Q^s_\Om(\R^n)$, we will also need the Hermitian form
\begin{equation}\label{eq:Besicovich_product1}
(f,g)_0:=\sum_{m\in\Z^M}\widehat{F}_m\overline{\big(\widehat{G}_m\big)},\quad f,g\in Q^s_\Om(\R^n).
\end{equation}
Note that the form \eqref{eq:Besicovich_product1} is bounded in $Q^s_\Om(\R^n)$, $s>M/2$.

\medskip

In addition to the space $Q^s_\Om(\R^n)$, $s>M/2$, in Section \ref{sec:spaces}, we define a finer scale 
of Hilbert subspaces 
\[
Q^{l,s}_\Om(\R^n)\subseteq Q^s_\Om(\R^n),\quad l\in\Z_{\ge 0}.
\] 
By definition, the space $Q^{l,s}_\Om(\R^n,\R^n)$ consists of maps $\R^n\to\R^n$ whose 
components are quasi-periodic functions in $Q^{l,s}_\Om(\R^n)$.
The main properties of $Q^{l,s}_\Om(\R^n)$ are discussed in Section \ref{sec:spaces}. 
In particular, it is shown that $Q^{l,s}_\Om(\R^n)$ is a {\em Banach algebra}.
In what follows we will omit the symbols $\R^n$ appearing in the notation of the space of quasi-periodic vector 
fields $Q^{l,s}_\Om(\R^n,\R^n)$ and write $Q^{l,s}_\Om$ instead. 
For $\rho>0$ denote by $B_{Q^{l,s}_\Om}(\rho)$ the open ball of radius 
$\rho$ in $Q^{l,s}_\Om$. In Section \ref{sec:euler} we prove the following theorem on 
the solutions of the Euler equation.

\begin{Th}\label{th:euler}
Assume that $s>M/2+1$ and $l\ge 2$. Then, for any $\rho>0$ there exists $T>0$ such that
for any divergence free $u_0\in B_{Q^{l,s}_\Om}(\rho)$ there exists a unique solution 
\begin{equation}\label{eq:class_od_solutions}
u\in C\big([-T,T],Q^{l,s}_\Om\big)\cap C^1\big([-T,T],Q^{l-1,s}_\Om\big)
\end{equation}
of the Euler equation \eqref{eq:euler} such that the pressure $\pp(t)$ belongs to $Q^s_\Om(\R^n)$ and 
has mean-value zero for any $t\in[-T,T]$. The solution depends continuously on the initial data in the sense that 
the data-to-solution map
\[
B_{Q^{l,s}_\Om}(\rho)\to C\big([-T,T],Q^{l,s}_\Om\big)\cap C^1\big([-T,T],Q^{l-1,s}_\Om\big),\quad
u_0\mapsto u,
\]
is continuous. In addition, we have that $\pp\in C\big([-T,T],Q^{l+1,s}_\Omega)$.
\end{Th}

\begin{Rem}\label{rem:external_force}
Theorem \ref{th:euler} continues to hold if we add a quasi-periodic external force 
$F\in C^1\big((-\infty,\infty),Q^{l+1,s}_\Om\big)$ such that $\Div F=0$ or $F=-\nabla U$ for some $U\in S'(\R^n)$ 
on the right side of \eqref{eq:euler}. The analytic dependence in Proposition \ref{prop:analyticity}, Section \ref{sec:euler}, 
continues to hold if we assume that $F\in C^1\big((-\infty,\infty),Q^{l+1,s}_{\Om,\bullet}\big)$ where $Q^{l+1,s}_{\Om,\bullet}$
consists of quasi-periodic functions with bounded set of Fourier exponents (see \eqref{eq:Q_bullet}).
\end{Rem}

The solution in Theorem \ref{th:euler} is unique in the class of solutions satisfying \eqref{eq:class_od_solutions}
such that the pressure $\pp(t)$ belongs to $Q^s_\Om(\R^n)$ for any $t\in[-T,T]$. 
The condition on the pressure can be replaced by another condition on the pressure (or, on the solution itself) but 
{\em cannot} be avoided, as seen from the transformation $\widetilde{u}(t):=u(t)+c t$ and 
$\widetilde{\pp}(t,x):=\pp(t,x)+(c,x)$, $c\in\R^n$, that transforms a solution $u$ of \eqref{eq:euler} with pressure $\pp$ to 
a solution $\widetilde{u}$ of \eqref{eq:euler} with pressure $\widetilde{\pp}$ and the same initial data $u_0$.
The transformation above can be easily modified so that the solution $\widetilde{u}$ will have a blow-up in finite time.

\medskip

\noindent{\em The total energy:} Since a quasi-periodic vector field $u\in Q^{l,s}_\Om\subseteq Q^s_\Om$, 
$l\ge 0$, $s>\frac{M}{2}+1$, does not decay at infinity, its total energy cannot be defined in the usual way.
Nevertheless, we can consider the {\em (averaged) energy}
\begin{equation}\label{eq:the_energy}
E(u):=\frac{1}{2}\lim_{T\to\infty}\frac{1}{(2T)^n}\int_{[-T,T]^n}(u,u)\,dx\ge 0
\end{equation}
which is well-defined and, as can be easily seen from the uniform convergence of the Fourier series,
$E(u)=\frac{1}{2}(u,u)_0$ where (cf. \eqref{eq:Besicovich_product1}),
\begin{equation}\label{eq:Besicovich_product2}
(u,v)_0:=\sum_{m\in\Z^M}(\widehat{u}_m,\overline{\widehat{v}_m}),\quad u,v\in Q^{l,s}_\Om,
\end{equation}
and $\widehat{u}_m,\widehat{v}_m\in\C^n$ are the Fourier coefficients of the vector fields $u,v\in Q^{l,s}_\Om$,
respectively (see Remark \ref{rem:many_things} in Section \ref{sec:spaces}). 
%Note that \eqref{eq:Besicovich_product2} is a bounded Hermitian form $Q^s_\Om\times Q^s_\Om\to\C$. 
A direct computation shows that the solutions in Theorem \ref{th:euler} {\em preserve} the averaged energy \eqref{eq:the_energy}.
In fact, since $u\in C\big([-T,T],Q^{l,s}_\Om\big)\cap C^1\big([-T,T],Q^{l-1,s}_\Om\big)$, we can differentiate
\eqref{eq:the_energy} (and \eqref{eq:Besicovich_product2}) in time to conclude from the Stokes' theorem 
and the fact that $u$ is divergence free that
\begin{align}
\dt E(u)&=\big(\dt u,u\big)_0
=\lim_{T\to\infty}\frac{1}{(2T)^n}\int_{[-T,T]^n}\Big(-\big(u\cdot\nabla u,u\big)-\big(\nabla\pp,u\big)\Big)\,dx\nonumber\\
&=\lim_{T\to\infty}\frac{1}{(2T)^n}\int_{[-T,T]^n}\big(\nabla{\rm U},u\big)\,dx=
\lim_{T\to\infty}\frac{1}{(2T)^n}\int_{[-T,T]^n}\Div\big({\rm U} u\big)\,dx\nonumber\\
&=\lim_{T\to\infty}\frac{1}{(2T)^n}\int_{\partial([-T,T]^n)}(u,\nu)\,{\rm U}\,d\sigma=0,\label{eq:energy_computation}
\end{align}
where $\nu$ is the outward unit normal to the boundary $\partial([-T,T]^n)$ of the cube $[-T,T]^n$ in $\R^n$,
$d\sigma$ is the surface volume form corresponding to the Euclidena metric in $\R^n$, 
and ${\rm U}:=-\frac{1}{2}(u,u)-\pp$.
The limit of the flux integral above vanishes since the integrand is uniformly bounded on $\R^n$ by the inclusion
$Q^{l,s}_\Om(\R^n)\subseteq C_b(\R^n)$ and the fact that $u,\pp\in Q^{l,s}_\Om$ (cf. Theorem \ref{th:euler}).

\begin{Rem}%\label{rem:variational_principle}
One can use the averaged energy \eqref{eq:the_energy} to define a right-invariant weak Riemannian
metric on the group of volume preserving quasi-periodic diffeomorphisms of $\R^n$ (cf. Section \ref{sec:diffeomorphisms}).
It can be shown that the solutions of the Euler equations in Theorem \ref{eq:euler} locally minimize 
the length of the curves corresponding to this metric. In particular, we see that the Euler equation can be derived from 
the variational principle applied to the averaged energy. We will discuss the Riemannian geometry of the group of 
volume preserving diffeomorphisms in detail in a separate work.
\end{Rem}

\medskip

\noindent{\em The particle trajectories:} Let us now take a divergence free quasi-periodic initial velocity field 
$u_0\in Q^{l,s}_\Om$ and let $u\in C\big((-T_1,T_2),Q^{l,s}_\Om\big)\cap C^1\big((-T_1,T_2),Q^{l-1,s}_\Om\big)$ be 
the solution of the Euler equation \eqref{eq:euler} on its maximal interval of existence $(-T_1,T_2)$, $T_1,T_2>0$.
(The existence of such an interval follows from Theorem \ref{th:euler}.)
The {\em trajectory (streamline) of a fluid particle} that is positioned at $x\in\R^n$ for $t=0$ is given by the integral curve of 
the (non-autonomous) ordinary differential equation $\dt\varphi(t)=u\big(t,\varphi(t)\big)$, $\varphi(t,x)|_{t=0}=x$ in $\R^n$.
It follows from Lemma \ref{lem:ode} in Section \ref{sec:euler} that these integral curves are defined for any $x\in\R^n$ 
and $t\in(-T_1,T_2)$ so that for any given $t\in(-T_1,T_2)$ the map $\varphi(t) : \R^n\to\R^n$, $x\mapsto\varphi(t,x)$ is 
a quasi-periodic diffeomorphism of $\R^n$. The quasi-periodic diffeomorphisms are quasi-periodic perturbations of 
the identity map in $\R^n$. They form a topological group and are studied in Section \ref{sec:diffeomorphisms} 
(cf. Theorem \ref{th:topological_group} and Theorem \ref{th:composition}). 
Proposition \ref{prop:analyticity} in Section \ref{sec:euler} implies the following corollary.

\begin{Coro}\label{coro:analyticity}
The trajectories of the fluid particles are analytic curves in $\R^n$ that depend analytically on 
the initial data $u_0\in Q^{l,s}_\Om$. 
\end{Coro}

\noindent The corollary extends the result in \cite{Shnir} from the periodic to the quasi-periodic setting.
In contrast to \cite{Serfati}, the particle trajectories in Corollary \ref{coro:analyticity} depend analytically on the initial data.

Another consequence of Proposition \ref{prop:analyticity} is that the Fourier coefficients of the fluid velocity are analytic functions 
of time and the initial data. Note that generically, the velocity field $u : (-T_1,T_2)\to Q^{l,s}_\Om$ is {\em not} real analytic. 
Moreover, its dependence on the initial data is not even Lipschitz continuous.

\begin{Coro}\label{coro:u_m-analytic}
Let $u\in C\big((-T_1,T_2),Q^{l,s}_\Om\big)\cap C^1\big((-T_1,T_2),Q^{l-1,s}_\Om\big)$ be
the solution of the Euler equation \eqref{eq:euler} on its maximal time of existence. 
Then, for any $m\in\Z^M$ the Fourier coefficient $\widehat{u}_m : (-T_1,T_2)\to\C^n$ is an analytic curve that 
depends analytically on the initial data $u_0\in Q^{l,s}_\Om$.
\end{Coro}

\medskip

\noindent{\em The phase space:} With any choice of the linear map $\Omega$ which satisfies 
the non-resonance condition {\rm (NC)} one associates a discrete lattice
$\Gamma_\Om\equiv\big\{\gamma\in\R^n\,\big|\,\Om(\gamma)\in\Z^M\big\}$ in $\R^n$.
The rank of $\Gamma_\Om$ can take any integer value in $\{0,...,n-1\}$ (cf. Lemma \ref{lem:Gamma_Om} in the Appendix). 
Hence, depending on the choice of $\Om$, the phase space of the fluid in Theorem \ref{th:euler} is diffeomorphic to 
the cylinders $\T^r\times\R^{n-r}$ where $r=\rk\Gamma_\Om$ and $\T:=\R/\Z$. Informally, we can say that when $r>0$ we 
have $r$ periodic and $n-r$ purely quasi-periodic directions in $\R^n$. If $r=0$ the phase space is diffeomorphic to $\R^n$ and 
does not have periodic directions. This implies that for such $\Om$ a quasi-periodic vector field $u\in Q^{l,s}_\Om$ 
in general position is {\em purely quasi-periodic} in the sense that it does {\em not} have non-vanishing periods.
By the uniqueness, if $u_0\in Q^{l,s}_\Om$ is purely quasi-periodic, then $u(t)\in  Q^{l,s}_\Om$ is purely quasi-periodic
for any $t$ in the interval of existence.
%%%%%%%%%%%%%%%%%%%%%%%Another argument%%%%%%%%%%%%%%%%%%%%%%%%%%%%
%It follows from Corollary \ref{coro:u_m-analytic} and Baire category theorem that {\em if the initial data $u_0\in Q^{l,s}_\Om$ is 
%purely quasi-periodic then there exists a dense $G_\delta$-set $\mathcal{N}\subseteq (-T_1,T_2)$ of full measure such that for any 
%$t\in\mathcal{N}$, $u(t)\in Q^{l,s}_\Om$ is purely quasi-periodic}.
%%%%%%%%%%%%%%%%%%%%%%%%%%%%%%%%%%%%%%%%%%%%%%%%%%%%%%%%%%%

\medskip

\noindent{\em Related work\&Discussion:} There are many important works related to the solutions of the Euler equation 
on $\R^n$ in various function spaces. Since we are not able in this short section to review even a small part of these works, 
we will mention only a few and will refer to the monographs \cite{Chemin,MB} for further references. 
Local existence and uniqueness in the Sobolev space $H^s(\R^n)$, $s>\frac{n}{2}+1$, is proved in \cite{Kato} (see also \cite{KatoPonce,KatoBook}). 
Concerning spatially non-decaying solutions, we mention \cite{Chemin,Serfati} (and the references therein) were local existence and uniqueness is 
proved in the H\"older space $C^\gamma_b(\R^n)$, $\gamma>1$. (Note that these solutions do not depend continuously on 
the initial data \cite{MY}.) Solutions in spaces of functions that grow at infinity were constructed recently in \cite{McOwenTopalov}. 
Almost-periodic solution of the Euler equation are considered in \cite{TTY,ST}.
In \cite{TTY} is proved that a unique weak solution of the 2d Euler equation in the Besov space $B^0_{\infty,1}(\R^2)$ with
almost-periodic initial data is almost-periodic for any time. In \cite{ST} is proved that if the initial data $u_0$ in the Besov space
$B^1_{\infty,1}(\R^n)$, $n\ge 2$, is almost-periodic in the (larger) Besov space $B^0_{\infty,1}(\R^n)$, then the unique solution 
in $B^1_{\infty,1}(\R^n)$ (\cite{PP}) is almost-periodic in $B^0_{\infty,1}(\R^n)$ for any $t$ in the interval of existence. 
The discrepancy of the norms appears since the continuity on the initial data in $B^1_{\infty,1}$ is established with respect to 
the (weaker) norm in $B^0_{\infty,1}(\R^n)$. The proofs in \cite{TTY,ST} are based on the Bochner's characterization of almost-periodic 
functions and on the continuity on the initial data in various norms -- an idea applied initially to the solutions of 
the Navier-Stokes equation in \cite{GMN}. We are not aware of works on quasi-periodic solutions of the (full) Euler equation. 
As mentioned above, quasi-periodic functions appear naturally in applications and have much more rigid structure than 
the almost-periodic ones. Note that Bochner's characterization does {\em not} apply to quasi-periodic functions, and hence
results on quasi-periodic functions are not readily available. Note also that our solutions depend continuously on the initial data 
with respect to the norm in $Q^{l,s}_\Om$. By Lemma \ref{lem:Bochner_property} in the Appendix the elements of $Q^{l,s}_\Om$ are 
almost-periodic in the sense of Bochner with respect to the norm in $Q^{l,s}_\Om$. 
Finally, note that results similar to the ones proved in this paper can be also proved in the viscose case.

\medskip

\noindent{\em Organization of the paper:} In Section \ref{sec:spaces} we define the spaces 
of quasi-periodic functions $Q^{l,s}_\Om(\R^n)$, $l\ge 0$, $s>M/2$, and study their main properties.
In particular, in Lemma \ref{lem:Q^{l,s}} and Proposition \ref{prop:Q^{l,s}} we discuss the Fourier series 
of quasi-periodic functions. In Section \ref{sec:diffeomorphisms} we define the group $QD^{l,s}_\Om(\R^n)$, 
$l\ge 0$, $s>\frac{M}{2}+1$, of quasi-periodic diffeomorphisms of $\R^n$ and prove that it is a topological group 
that enjoys additional regularity properties formulated in Theorem \ref{th:composition}. (Part of the results in 
Section \ref{sec:diffeomorphisms} are contained in the first author's PhD thesis.)
Theorem \ref{th:euler} is proved in Section \ref{sec:euler}. The main obstacle in this section is to prove analyticity 
and to avoid the appearance of {\em small denominators} in the Lagrangian representation of the Euler equation 
(cf. \cite{Arnold,EM,Wol}). This is achieved by the decomposition \eqref{eq:F}, Proposition \ref{prop:pde<->ode}, 
and a sequence of lemmas concerning the analyticity of the non-linear maps (between spaces of quasi-periodic functions) 
appearing as factors in the decomposition \eqref{eq:F}. The paper has an Appendix where we prove several technical lemmas.

\medskip

\noindent{\em Acknowledgment:} Our dear college and long time co-author Thomas Kappeler was 
involved at an earlier stage of this project. His constant encouragement, insights and influence cannot be overstated.
The authors are also thankful to Jean-Claude Saut for referring them to G\'erard Iooss' papers on quasipatterns.

\section{Spaces of quasi-periodic functions}\label{sec:spaces}
In this section we introduce and study in detail the scale of Hilbert spaces of quasi-periodic functions
$Q^{l,s}_\Om(\R^n)$, $l\ge\Z_{\ge 0}$, $s>M/2$.

Any $f\in Q^s_\Om(\R^n)\subseteq L^\infty(\R^n)$ defines a tempered distribution in $S'(\R^n)$ which we identify with $f$.
In view of the uniform convergence, the series \eqref{eq:f_expansion} converges to $f$ in $S'(\R^n)$ in distributional sense.
Recall that a series $\sum_{j\in J} f_j$, $f_j\in S'(\R^n)$, where $J$ is a countable set of indices, converges to 
$f\in S'(\R^n)$ in $S'(\R^n)$ {\em independently of the order of summation} (or equivalently, {\em unconditionally}) if 
for any test function $\varphi\in S(\R^n)$ the series $\sum_{j\in J}\langle f_j,\varphi\rangle$ converges unconditionally to 
$\langle f,\varphi\rangle$, i.e., for any bijection $\sigma : J\to J$ the series $\sum_{j\in J}\langle f_{\sigma(j)},\varphi\rangle$
converges to $\langle f,\varphi\rangle$. We have the following characterization of the image of the embedding
$Q^s_\Om(\R^n)\subseteq S'(\R^n)$.

\begin{Lem}\label{lem:qp_distributions1}
Assume that $s>M/2$. A distribution $f\in S'(\R^n)$ belongs to $Q^s_\Om(\R^n)$ if and only if $f$ can be written as 
a convergent in $S'(\R^n)$ series,
\begin{equation}\label{eq:f_expansion1}
f(x)=\sum_{m\in \Z^M}\widehat{f}_m e^{i (\Lambda_m,x)},\quad
\Lambda_m\equiv2\pi\Om^T(m),
\end{equation}
such that $(\widehat{f}_m)_{m\in\Z^M}$ is a sequence of complex numbers such that
\begin{equation}\label{eq:f_condition1}
\sum_{m\in\Z^M}|\widehat{f}_m|^2\m^{2s}<\infty. 
\end{equation}
If the conditions \eqref{eq:f_expansion1} and \eqref{eq:f_condition1} above hold then \eqref{eq:f_expansion1} 
converges absolutely (and uniformly) to $f$ and, in particular, \eqref{eq:f_expansion1} converges in $S'(\R^n)$
independently of the order of summation. Moreover, 
\begin{equation}\label{eq:f_m}
\widehat{f}_m=\lim_{T\to\infty}\frac{1}{(2T)^n}\int_{[-T,T]^n}f(x)\,e^{- i(\Lambda_m,x)}\,dx
\end{equation}
and for any $m\in\Z^M$ the coefficient $\widehat{f}_m$ coincides with the Fourier coefficinent
$\widehat{F}_m$ of the periodic function $F\in C(\T^M)$ (see Definition \ref{def:qp}).
\end{Lem}

Lemma \ref{lem:qp_distributions1} will be often applied together with Lemma \ref{lem:uniqueness} in the Appendix.

\begin{Rem}\label{rem:many_things}
Note that the expansion \eqref{eq:f_expansion1} of an element $f\in Q^s_\Om(\R^n)$ such that \eqref{eq:f_condition1} 
holds is unique since the coefficients $\widehat{f}_m$  can be determined from $f$ by formula \eqref{eq:f_m}. 
We will refer to \eqref{eq:f_expansion1} as the {\em Fourier series} and to $\widehat{f}_m$, $m\in\Z^M$, as the 
{\em Fourier coefficients} of the quasi-periodic function $f\in Q^s_\Om(\R^n)$.
Since the coefficients $\widehat{f}_m$, $m\in\Z^M$, satisfy \eqref{eq:f_m} the terminology is consistent with the one
used in the theory of almost-periodic functions (see e.g. \cite{Levitan}).
%Note alsuch that in our fraimwork we do not use the fact that an almost-periodic function is uniquely determined by its Fourier series
%(see e.g. \cite[Ch. I, \S 5]{Levitan}).
\end{Rem}

%\begin{Rem}\label{rem:notation}
%For simplicity of notation, in what follows we denote the Fourier coefficient of a quasi-periodic function $f\in Q^s_\Om(\R^n)$ by 
%$\widehat{f}_m$, $m\in\Z^M$. One should keep in mind that by Lemma \ref{lem:qp_distributions1} these are also the Fourier coefficients 
%$\widehat{F}_m$ of the periodic function $F\in C(\T^M)$ (see Definition \ref{def:qp}).
%\end{Rem}

\begin{Rem}\label{rem:L^infty}
It follows from the Cauchy-Schwatz inequality and representation \eqref{eq:f_expansion1} and \eqref{eq:f_condition1} that 
there exists $C\equiv C_s>0$ such that for any $f,g\in Q^s_\Om(\R^n)$, $|f-g|_\infty\le C\|f-g\|_s$, where $|\cdot|_\infty$
denotes the norm in $L^\infty(\R^n)$. In particular, we see that the inclusions $Q^s_\Om(\R^n)\subseteq C_{ap}(\R^n)$ and 
$Q^s_\Om(\R^n)\subseteq L^\infty(\R^n)$ are bounded.
\end{Rem}

\begin{proof}[Proof of Lemma \ref{lem:qp_distributions1}]
The proof of this Lemma is straightforward. In fact, assume that $f\in Q^s_\Om(\R^n)$. Then, by definition, $f(x)=F(\Om_{\tt p}(x))$ 
where $F\in H^s(\T^m)$ whit $s>\frac{M}{2}$. In particular, inequality \eqref{eq:f_condition1} holds with 
$\widehat{f}_m$ replaced by $\widehat{F}_m$ for any $m\in\Z^M$. This together with the Cauchy-Schwarz inequality implies that 
$\sum_{m\in \Z^M}|\widehat{F}_m|<\infty$. Hence, the Fourier series $\sum_{m\in\Z^M}\widehat{F}_m e^{2\pi i(m,y)}$ converges 
uniformly and absolutely to $F$ on the torus $\T^M$. 
We have
$f(x)=F(\Om_{\tt p}(x))=\sum_{m\in\Z^M}\widehat{F}_m e^{2\pi i(m,\Om(x))}=
\sum_{m\in\Z^M}\widehat{F}_m e^{i(2\pi\Om^T(m),x)}$ 
where the series converge uniformly and absolutely.
In particular, we see that the statement of the Lemma holds with $\widehat{f}_m=\widehat{F}_m$, $m\in\Z^m$.

Conversely, assume that \eqref{eq:f_expansion1} converges in $S'(\R^n)$ to some $f\in S'(\R^n)$.
Assume in addition that  \eqref{eq:f_condition1} holds. Then, by the Cauchy-Schwarz inequality, 
$\sum_{m\in \Z^M}|\widehat{f}_m|<\infty$. This implies that the series $\sum_{m\in \Z^M}\widehat{f}_m e^{i (\Lambda_m,x)}$
converges uniformly and absolutely to $f$ and $f\in C_b(\R^n)$.
Hence, we have 
\begin{eqnarray*}
f(x)=\sum_{m\in \Z^M}\widehat{f}_m e^{i (\Lambda_m,x)}=\sum_{m\in \Z^M}\widehat{f}_m e^{2\pi i (m,\Om(x))}=
F(\Om_{\tt p}(x))
\end{eqnarray*} 
where $F(y):=\sum_{m\in\Z^M}\widehat{f}_m e^{2\pi i(m,y)}$ converges uniformly and \eqref{eq:f_condition1} holds.
This implies that $F\in H^s(\T^M)$ and its Fourier coefficients coincide with the coefficients $\widehat{f}_m$, $m\in\Z^m$. 
Combining the above we conclude that $f\in Q^s_\Om(\R^n)$.

Finally, note that since the series \eqref{eq:f_expansion1} converges to $f$ absolutely the sum is independent of the order 
of summation in $m\in\Z^M$. The uniform convergence of \eqref{eq:f_expansion1} implies \eqref{eq:f_m}.
\end{proof}

\medskip\medskip

\noindent{\em The space $Q^{l,s}_\Om(\R^n)$}:
For given $l\in\Z_{\ge 0}$ and $s>\frac{M}{2}$ we will introduce a finer scale $Q^{l,s}_\Om(\R^n)$ of Sobolev spaces 
of quasi-periodic functions,
\begin{equation}\label{eq:Q^{l,s}}
Q^{l,s}_\Om(\R^n):=\big\{f\in Q^s_\Om(\R^n)\,\big|\, \p_x^\beta f\in Q^s_\Om(\R^n),\,|\beta|\le l\big\}
\end{equation}
where $\beta\in\Z_{\ge 0}^n$ is a multi-index and $\p_x^\beta\equiv\p_{x_1}^{\beta_1}\cdots\p_{x_n}^{\beta_n}$ where 
$\p_{x_k}$ denotes the distributional partial derivative in the $x_k$ variable.
We define the norm in $Q^{l,s}_\Om(\R^n)$,
\begin{equation}\label{eq:Q^{l,s}-norm}
|f|_{l,s}:=\Big(\sum_{|\beta|\le l}\|\p_x^\beta f\|_s^2\Big)^{1/2}.
\end{equation}
Using Lemma \ref{lem:qp_distributions1} and Lemma \ref{lem:uniqueness} one can easily prove the following

\begin{Lem}\label{lem:Q^{l,s}}
Assume that $l\in\Z_{\ge 0}$ and $s>\frac{M}{2}$. Then the following statements hold:
\begin{itemize}
\item[(i)] The space $Q^{l,s}_\Om(\R^n)$ equipped with the norm \eqref{eq:Q^{l,s}-norm} is a Hilbert space.
\item[(ii)] An element $f\in Q^s_\Om(\R^n)$ belongs to $Q^{l,s}_\Om(\R^n)$ if and only if its Fourier coefficients 
(see Remark \ref{rem:many_things}) satisfy
\begin{equation}\label{eq:boundedness_Q^{l,s}-norm}
\sum_{m\in\Z^M}|\widehat{f}_m|^2\langle\Lambda_m\rangle^{2l}\m^{2s}<\infty,\quad
\langle\Lambda_m\rangle\equiv\sqrt{1+|\Lambda_m|^2},
\end{equation}
where  $|\Lambda_m|\equiv\sqrt{\sum_{j=1}^n\Lambda_{m,j}^2}$ and 
$\big(\Lambda_{m,1},...,\Lambda_{m,n}\big)$ are the components of the Fourier exponent $\Lambda_m\in\R^n$.
Moreover, for $f\in Q^{l,s}_\Om(\R^n)$  we have that $\widehat{(\p_x^\beta f)}_m=(i\Lambda_m)^\beta\widehat{f}_m$ 
for any multi-index $\beta$ with $0\le|\beta|\le l$ and for any $m\in\Z^M$.
\item[(iii)] The norms \eqref{eq:Q^{l,s}-norm} and
\begin{equation}\label{eq:Q^{l,s}'-norm}
\|f\|_{l,s}:=\Big(\sum_{m\in\Z^M}|\widehat{f}_m|^2\langle\Lambda_m\rangle^{2l}\m^{2s}\Big)^{1/2}
\end{equation}
in $Q^{l,s}_\Om(\R^n)$ are equivalent.
\end{itemize}
\end{Lem}

\begin{Rem}
By Lemma \ref{lem:Q^{l,s}}, the scalar product in $Q^{l,s}_\Om(\R^n)$ is equivalent to
\begin{equation}\label{eq:Q^{l,s}-scalar_product}
(f,g)_{s,l}:=\sum_{m\in\Z^M}\widehat{f}_m\overline{\big(\widehat{g}_m\big)}\langle\Lambda_m\rangle^{2l}\m^{2s},
\quad f,g\in Q^{l,s}_\Om(\R^n).
\end{equation}
\end{Rem}

\begin{Rem}
The space of quasi-periodic functions $Q^{l,s}_\Om(\R^n)$ appears as a natural generalization of the Sobolev space of functions 
on the torus $\T^n$. In fact, if $M=n$ and $\Omega$ is the identity matrix $\Id_{n\times n}$ then $Q^{l,s}_\Om(\R^n)$ coincides 
with the Sobolev space $H^{s+l}(\T^n)$ interpreted as a space of $\Z^n$-periodic functions in $\R^n$.
\end{Rem}

The corollary below follows directly from \eqref{eq:f_expansion1} and Lemma \ref{lem:Q^{l,s}} (iii).

\begin{Coro}\label{coro:density}
The set of finite linear combinations of exponents $e^{i (\Lambda_m,x)}$, $m\in\Z^M$, is dense in $Q^{l,s}_\Om(\R^n)$.
\end{Coro}

Let us now proof Lemma \ref{lem:Q^{l,s}}.

\begin{proof}[Proof of Lemma \ref{lem:Q^{l,s}}]
The proof of item $(i)$ is straightforward. Let $(f_j)_{j\ge 1}$ be a Cauchy sequence in $Q^{l,s}_\Om(\R^n)$. Then,
in view of \eqref{eq:Q^{l,s}-norm}, the sequence $(\p_x^\beta f_j)_{j\ge 1}$ is a Cauchy sequence in $Q^s_\Om(\R^n)$
for any given multi-index $\beta$, $|\beta|\le l$. As the space $Q^s_\Om(\R^n)$ is complete we conclude that there exists
$v_\beta\in Q^s_\Om(\R^n)$ such that 
\begin{equation}\label{eq:convergence_in_Q}
\p_x^\beta f_j\stackrel{Q^s_\Om}{\longrightarrow} v_\beta,\quad j\to\infty.
\end{equation}
On the other side, the continuity of the inclusion $Q^s_\Om(\R^n)\subseteq L^\infty(\R^n)$ (see Remark \ref{rem:L^infty})
and the fact that $f_j\stackrel{Q^s_\Om}{\longrightarrow} v_0$ as $j\to\infty$ implies that $f_j\to v_0$ in $S'(\R^n)$.
Hence, $\p_x^\beta f_j\stackrel{S'}{\longrightarrow} \p_x^\beta v_0$ as $j\to\infty$ for any $|\beta|\le l$.
By comparing this with \eqref{eq:convergence_in_Q} we conclude that $\p_x^\beta v_0=v_\beta\in Q^s_\Om(\R^n)$ for any 
$|\beta|\le l$. Hence, $v_0\in Q^s_\Om(\R^n)$ and (by \eqref{eq:convergence_in_Q})
$f_j\stackrel{Q^s_\Om}{\longrightarrow}v_0$ as $j\to\infty$. This completes the proof of $(i)$.

In order to prove item $(ii)$ take $f\in Q^{l,s}_\Om(\R^n)\subseteq Q^s_\Om(\R^n)$.
Then, by Lemma \ref{lem:qp_distributions1},
$f(x)=\sum_{m\in\Z^M}\widehat{f}_m e^{i (\Lambda_m,x)}$
where the series converges to $f$ in $S'(\R^n)$ independently of the order of summation.
This implies that for any $\beta$ with $|\beta|\le l$,
\begin{equation}\label{eq:expansion_derivative}
\p_x^\beta f=\sum_{m\in\Z^M}\widehat{f}_m (i\Lambda_m)^\beta e^{i (\Lambda_m,x)}
\end{equation}
where the series converges to $\p_x^\beta f$ in $S'(\R^n)$ independently of the order of summation.
It then follows from Lemma \ref{lem:uniqueness} that \eqref{eq:expansion_derivative} is the Fourier expansion of 
the quasi-periodic function $\p_x^\beta f\in Q^s_\Om(\R^n)$ (see Remark \ref{rem:many_things}). In particular, we see that 
for any multi-index $\beta$ with $|\beta|\le l$, $\widehat{\big(\partial_x^\beta f\big)}_m=(i\Lambda_m)^\beta\widehat{f}_m$, 
and by Lemma \ref{lem:qp_distributions1},
\begin{equation}\label{eq:boundedness_Q^{l,s}-pre-norm}
\sum_{m\in\Z^M}|\widehat{f}_m|^2 \big(|\Lambda_{m,1}|^{2\beta_1}\cdots|\Lambda_{m,n}|^{2\beta_n}\big)\m^{2s}<\infty,
\end{equation}
where $\Lambda_m\equiv(\Lambda_{m,1},...,\Lambda_{m,n})$.
Since inequality \eqref{eq:boundedness_Q^{l,s}-pre-norm} holds for any $\beta$ with $|\beta|\le l$, we sum up inequalities
of the from \eqref{eq:boundedness_Q^{l,s}-pre-norm} for different values of $\beta$ to conclude that
\[
\sum_{m\in\Z^M}|\widehat{f}_m|^2(1+|\Lambda_m|_1)^{2l}\m^{2s}<\infty
\]
where $|\Lambda_m|_1\equiv\sum_{j=1}^n|\Lambda_{m,j}|$. The last inequality is equivalent to 
\eqref{eq:boundedness_Q^{l,s}-norm}.

Item $(iii)$ follows easily from the arguments used to prove $(ii)$.
\end{proof}

\medskip

Denote by $C^k_b(\R^n)$, $k\in\Z_{\ge 0}$, the space of continuously differentiable functions on $\R^n$ whose partial derivatives 
of order $\le k$ are continuous and bounded in $\R^n$. Similarly, we denote by $C^k(\R^n)$, $k\in\Z_{\ge 0}$, the space of continuously
differentiable functions on $\R^n$ whose partial derivatives of order $\le k$ are continuous in $\R^n$ (and not necessarily bounded).
We have the following

\begin{Prop}\label{prop:Q^{l,s}}
Assume that $l\in \Z_{\ge 0}$ and $s>\frac{M}{2}$. Then we have:
\begin{itemize} 
\item[(i)]  For any multi-index $\beta$ the mapping 
\[
\p_x^\beta : Q^{l+|\beta|,s}_\Om(\R^n)\to Q^{l,s}_\Om(\R^n)
\] 
is continuous. 
\item[(ii)] The space $Q^{l,s}_\Om(\R^n)$ is a Banach algebra with respect to the pointwise multiplication of functions.
\item[(iii)] If $s>\frac{M}{2}+k$ for some $k\in\Z_{\ge 0}$ then $Q^{l,s}_\Om(\R^n)\subseteq C^{k+l}_b(\R^n)$
and the inclusion is continuous. More generally, 
\begin{equation}\label{eq:H_loc}
Q^{l,s}_\Om(\R^n)\subseteq H^{\frac{n}{2}+\big(s-\frac{M}{2}\big)+l}_{loc}(\R^n)\cap C^{k+l}_b(\R^n).
\end{equation}
\end{itemize}
\end{Prop}

\begin{Rem}\label{rem:banach_algebra}
In the definition of a Banach algebra $(X,\|\cdot\|)$ we assume that the ring inequality $\|fg\|\le C\,\|f\|\|g\|$, $f,g\in X$, 
holds with a constant $C>0$ that is {\em not necessarily} equal to one.
\end{Rem}

\begin{proof}[Proof of Proposition \ref{prop:Q^{l,s}}]
Item $(i)$ follows directly from Lemma \ref{lem:qp_distributions1} and Lemma \ref{lem:Q^{l,s}} $(ii)$.
We will now prove $(ii)$ by iduction in $l\ge 0$: Assume that $l=0$ and take $f,g\in Q^{0,s}_\Om(\R^n)\equiv Q^s_\Om(\R^n)$.
Then, by definition $f=\Om_{\tt p}^*(F)$ and $g=\Om_{\tt p}^*(G)$ where $F,G\in H^s(\T^M)$. Clearly,
\begin{equation}\label{eq:pull-back}
fg=\Om_{\tt p}^*(FG).
\end{equation}
Since for $s>\frac{M}{2}$ the space $H^s(\T^M)$ is a Banach algebra, we then conclude that $FG\in H^s(\T^M)$ and hence
$fg\in Q^s_\Om(\R^n)$. The continuity of the pointwise multiplication of functions in $Q^s_\Om(\R^n)$ then follows from 
\eqref{eq:pull-back} and Lemma \ref{lem:pull-back}.
Further, assume that $l\ge 1$ and $Q^{l-1,s}_\Om(\R^n)$ is a Banach algebra.
Take $f,g\in Q^{l,s}_\Om(\R^n)$. Then, a simple approximation argument involving Corollary \ref{coro:density},
the induction hypothesis and item $(i)$ implies that for any $1\le j\le n$ we have
\begin{equation}\label{eq:product_formula}
\p_{x_j}(fg)=(\p_{x_j}f)g+f(\p_{x_j}g)
\end{equation}
in $S'(\R^n)$.
Using the induction hypothesis one more time one concludes from \eqref{eq:product_formula} that 
$\p_{x_j}(fg)\in Q^{l-1,s}_\Om(\R^n)$. This implies that $fg\in Q^{l,s}_\Om(\R^n)$. The continuity of the pointwise 
multiplication of functions in $Q^{l,s}_\Om(\R^n)$ then follows since, by \eqref{eq:product_formula} and the induction 
hypothesis, the map
\[
Q^{l,s}_\Om(\R^n)\times Q^{l,s}_\Om(\R^n)\to Q^{l-1,s}_\Om(\R^n),\quad(f,g)\mapsto\p_{x_j}(fg),
\]
is continuous for any $1\le j\le n$. This completes the proof of item $(ii)$.

Towards the proof of item $(iii)$, we first note that the embedding $Q^{l,s}_\Om(\R^n)\subseteq C^k_b(\R^n)$ follows directly
from the Sobolev embedding $H^s(\T^M)\subseteq C^k(\T^M)$, $s>\frac{M}{2}+k$, since any element $f\in Q^{l,s}_\Om(\R^n)$ 
has the form $f=\Om_{\tt p}^*(F)$ where $F\in H^s(\T^M)\subseteq C^k(\T^M)$. More generally, one sees from the trace theorem that
for any multi-index $\beta\in\Z_{\ge 0}^n$ with $|\beta|\le l$ one has
\[
\p_x^\beta f\in Q^s_\Om(\R^n)\subseteq H^{s-\frac{M-n}{2}}_{loc}(\R^n)
\]
which implies that $f\in H^{\frac{n}{2}+\big(s-\frac{M}{2}\big)+l}_{loc}(\R^n)\subseteq C^{k+l}(\R^n)$.
Combining this with the fact that for any multi-index $\beta$ with $|\beta|\le l$,
\[
\p_x^\beta f\in Q^s_\Om(\R^n)\subseteq C^k_b(\R^n),
\]
we conclude that $f\in C^{k+l}_b(\R^n)$. Since all of the inclusions above are continuous we conclude that
the inclusion $Q^{l,s}_\Om(\R^n)\subseteq C^{k+l}_b(\R^n)$ is continuous.
\end{proof}

\begin{Rem}\label{rem:complex_spaces}
In addition to the spaces $Q^{l,s}_\Om(\R^n)$, $l\in\Z_{\ge 0}$, $s>M/2$, 
that consist of real valued functions we will also consider complex spaces
$Q^{l,s}_{\Om,\C}(\R^n)\equiv Q^{l,s}_\Om(\R^n)\oplus i\,Q^{l,s}_\Om(\R^n)$
that consist of complex valued quasi-periodic functions. Note that all statements proved in this section hold also for
$Q^{l,s}_{\Om,\C}(\R^n)$. In particular, $Q^{l,s}_{\Om,\C}(\R^n)$ is a Hilbert space and a Banach algebra.
\end{Rem}

\medskip\medskip

The space of quasi-periodic functions $Q^{l,s}_\Om(\R^n)$ is closely related to the following space of periodic functions
on the torus $\T^M$. For given $l\in\Z_{\ge 0}$ and $s>\frac{M}{2}$ define
\begin{equation}\label{eq:H^{l,s}}
H^{l,s}_\Om(\T^M):=\big\{F\in H^s(\T^M)\,\big|\,\p_\Om^\beta F\in H^s(\T^M)\big\}
\end{equation}
where $\beta$ is a multi-index and 
\[
\p_\Om^\beta\equiv\p_{\Om_1}^{\beta_1}\cdots\p_{\Om_n}^{\beta_n}
\]
where $\p_{\Om_k}:=\sum_{j=1}^M\Om^j_k\p_{y_j}$ denotes the distributional derivative in 
the direction of the $k$-th column of the matrix of $\Om$. 
In the same way as above we define the norm in $H^{l,s}_\Om(\T^M)$,
\begin{equation}\label{eq:H^{l,s}-norm}
|F|_{l,s}:=\Big(\sum_{|\beta|\le l}\|\p_\Om^\beta F\|_s^2\Big)^{1/2}.
\end{equation}
(Note that we use the same symbols for the norms in $Q^{l,s}_\Om$ and $H^{l,s}_\Om$.)
One easily sees that $F\in H^{l,s}_\Om(\T^M)$ if and only if its Fourier coefficients satisfy
$\sum_{m\in\Z^M}|\widehat{F}_m|^2\langle\Lambda_m\rangle^{2l}\m^{2s}<\infty$. Note that 
$\Lambda_m=(\Lambda_{m,1},...,\Lambda_{m,n})$ where the component of $\Lambda_{m,k}$ ($1\le k\le n$) is equal to
$2\pi\sum_{j=1}^M\Om_k^jm_j$. 
We have

\begin{Prop}
Assume that $l\in\mathbb{Z}_{\ge 0}$ and $s>\frac{M}{2}$.
\begin{itemize}
\item[(i)] For any $l\in \Z_{\ge 0}$ and for any multi-index $\beta$ the mapping 
\[
\p_\Om^\beta : H^{l+|\beta|,s}_\Om(\T^M)\to H^{l,s}_\Om(\T^M)
\] 
is continuous.
\item[(ii)]The space $H^{l,s}_\Om(\T^M)$ is a Banach algebra with respect to the pointwise multiplication of functions.
\item[(iii)] If $s>\frac{M}{2}+k$ for some $k\in\Z_{\ge 0}$ then 
\[
H^{l,s}_\Om(\T^M)\subseteq C^k(\T^M).
\]
\end{itemize}
\end{Prop}

\begin{Lem}
Assume that $l\in\mathbb{Z}_{\ge 0}$ and $s>\frac{M}{2}$.
The map $H^{l,s}_\Om(\T^M)\stackrel{\Om_{\tt p}^*}{\to} Q^{l,s}_\Om(\R^n)$, $F\mapsto\Om_{\tt p}^*(F)$,
is a Banach algebras isomorphism.
\end{Lem}

The proofs of these two statements are similar to the proofs of the corresponding results for $Q^{l,s}_\Om(\R^n)$ and 
will be omitted.

\section{Quasi-periodic diffeomorphisms}\label{sec:diffeomorphisms}
In this Section we define the group of quasi-periodic diffeomorphism of $\R^n$ and study its properties.
For given $l\in\Z_{\ge 0}$ and $s>\frac{M}{2}+1$ consider the space $Q^{l,s}_\Om\equiv Q^{l,s}_\Om(\R^n,\R^n)$ of 
quasi-periodic vector fields on $\R^n$ where for simplicity we will often omit the symbols $\R^n$ appearing in the notation. 
(In order to avoid confusion we reserve the notation $Q^{l,s}_\Om(\R^n)$ for single valued quasi-periodic
functions only.) By Proposition \ref{prop:Q^{l,s}}, the inclusion
\begin{equation}\label{eq:C^1_b}
Q^{l,s}_\Om(\R^n)\subseteq C^1_b(\R^n)
\end{equation}
is continuous. This allows us to define the set of maps $\R^n\to\R^n$,
\begin{equation}\label{eq:QD^{l,s}}
QD^{l,s}_\Om(\R^n):=\big\{\varphi(x)=x+f(x), f\in Q^{l,s}_\Om\,\big|\,\exists\,\varepsilon_0>0\,\,\,\text{\rm s.t.}\,
\det\big(\Id+[df]\big)>\varepsilon_0\}
\end{equation}
where $\Id\equiv\Id_{n\times n}$ denotes the identity $n\times n$-matrix and $[df]$ is the Jacobian matrix of the map $f :\R^n\to\R^n$.
Note that by Hadamard's theorem (see e.g. \cite[Supplement 2.5D]{AMR}) and the fact that the components of the Jacobian matrix
$[d_xf]$ belong to $C_b$, the set $QD^{l,s}_\Om(\R^n)$ consists of orientation preserving $C^1$-diffeomorphisms of $\R^n$. 
Note alsuch that in general, $\varepsilon_0>0$ appearing in \eqref{eq:QD^{l,s}} depends on the choice of $\varphi$.  
Since, in view of the continuity of the inclusion \eqref{eq:C^1_b} and the Banach algebra property of $C_b(\R^n)$, 
the inequality appearing in \eqref{eq:QD^{l,s}} is an open condition in $Q^{l,s}_\Om(\R^n,\R^n)$, we conclude that 
$QD^{l,s}_\Om(\R^n)$ can be identified with an open set in $Q^{l,s}_\Om(\R^n,\R^n)$ that is coordinatized by 
$f\in Q^{l,s}_\Om(\R^n,\R^n)$. By definition, $QD^s_\Om(\R^n):=QD^{0,s}_\Om(\R^n)$. We have

\begin{Lem}\label{lem:differentiable_structure}
Assume that $l\in\Z_{\ge 0}$ and $s>\frac{M}{2}+1$.
Then the set $QD^{l,s}_\Om(\R^n)$ is a Banach manifold modeled on the space of quasi-periodic maps $Q^{l,s}_\Om(\R^n,\R^n)$.
\end{Lem}

We will first prove that $QD^{l,s}_\Om(\R^n)$ is a topological group. To this end, we prove

\begin{Prop}\label{prop:composition}
Assume that $l\in\Z_{\ge 0}$ and $s>\frac{M}{2}+1$.
Then the map
\[
Q^{l,s}_\Om(\R^n)\times QD^{l,s}_\Om(\R^n)\to Q^{l,s}_\Om(\R^n),\quad (g,\varphi)\mapsto g\circ\varphi,
\]
is continuous.
\end{Prop}

We start with a preparation:
Take $g\in Q^s_\Om(\R^n)$ and $\varphi\in QD^s_\Om(\R^n)$. Then for any given $x\in\R^n$ we have
$g(x)=G(\Om_{\tt p}(x))$ and $\varphi(x)=x+F(\Om_{\tt p}(x))$ where $G\in H^s(\T^M,\R)$ and $F\in H^s(\T^M,\R^n)$.
This implies that for any $x\in\R^n$,
\begin{eqnarray}
(g\circ\varphi)(x)&=G\Big(\Om_{\tt p}\Big(x+F\big(\Om_{\tt p}(x)\big)\Big)\Big)\nonumber\\
&=G\big(\Om_{\tt p}(x)+\Om_{\tt p}F(\Om_{\tt p}(x))\big)\nonumber\\
&=\big(G\circ\Phi\big)(\Om_{\tt p}(x))\label{eq:main1}
\end{eqnarray}
where
\begin{equation}\label{eq:Phi}
\Phi: \T^M\to\T^M,\quad\Phi :\theta\mapsto\theta+\Om_{\tt p}F(\theta).
\end{equation}
Denote
\[
D^s(\T^M):=\big\{\varphi\in\text{\rm Diff}^1_+(\T^M)\,\big|\,\varphi\in H^s(\T^M,\T^M)\big\}
\]
where $\text{\rm Diff}^1_+(\T^M)$ is the group of orientation preserving $C^1$-diffeomorphisms of the torus $\T^M$
and $H^s(\T^M,\T^M)$ is the space of maps $\T^M\to\T^M$ of Sobolev class $s$.
It is known that $D^s(\T^M)$ is a topological group  -- see e.g. \cite[Theorem 1.2]{IKT},\cite{EM}. 
We will prove

\begin{Lem}\label{lem:Phi-diffeomorphism}
Assume that $\varphi\in QD^s_\Om(\R^n)$, $s>\frac{M}{2}+1$, and let $\varphi(x)=x+F\big(\Om_{\tt p}(x)\big)$ where
$F\in H^s(\T^M,\R^n)$. Then, the map \eqref{eq:Phi} is a diffeomorphism in $D^s(\T^M)$ such that the diagram 
\begin{equation}\label{eq:diagram}
\begin{tikzcd}
\T^M\arrow{r}{\Phi}&\T^M\\
\R^n\arrow{u}{\Om_{\tt p}}\arrow{r}{\varphi}&\R^n\arrow[swap]{u}{\Om_{\tt p}}
\end{tikzcd}
\end{equation}
is commutative.
\end{Lem}

\begin{Rem}\label{rem:uniqueness}
One sees from the non-resonance condition {\rm (NC)} that a $C^1$-diffeomorphism $\Phi : \R^n\to\R^n$ which
satisfies \eqref{eq:diagram} is uniquely determined by $\varphi\in\text{\rm Diff}^1_+(\R^n)$.
\end{Rem}

Lemma \ref{lem:Phi-diffeomorphism} and formula \eqref{eq:Phi} imply that the map
\begin{equation}\label{eq:homomorphism}
h : QD^s_\Om(\R^n)\to D^s(\T^M),\quad\varphi\mapsto\Phi,
\end{equation}
is well-defined and continuous. In fact, see will see below that \eqref{eq:homomorphism} is a homomorphism of topological 
groups (Proposition \ref{prop:homomorphism}).

\begin{Rem}\label{rem:homomorphism}
Note that the correspondence \eqref{eq:homomorphism} is {\em not} necessarily injective. In fact,  since $\varphi$ is coordinatized by
the quasi-periodic function $f=\Om_{\tt p}^*(F)$, which by Lemma \ref{lem:pull-back} is uniquely determined by $F\in H^s(\T^M,\R^n)$, 
one sees from \eqref{eq:Phi} that the functions $F$ and $F+\gamma$ where $\gamma\in\Gamma_\Om$ (see \eqref{eq:Gamma_Om})
lead to the same diffeomorphism $\Phi\in D^s(\R^n)$. The arguments in the proof of Lemma \ref{lem:Phi-diffeomorphism} imply that 
$h(\varphi)=h(\psi)$ if and only if for any $x\in\R^n$, $\varphi(x)=\psi(x)+\gamma$ for some 
$\gamma\in\Gamma_\Om$ independent of $x$.
\end{Rem}

\begin{proof}[Proof of Lemma \ref{lem:Phi-diffeomorphism}]
Take $\varphi\in QD^s_\Om(\R^n)$, $s>\frac{M}{2}+1$ as in stated in the lemma.
Since by the Sobolev embedding $F\in H^s(\T^M,\R^n)\subseteq C^1(\T^M,\R^n)$ and 
$\varphi(x)=x+F(\Om_{\tt p}(x))$ we conclude that for any $x\in\R^n$
$[d_x\varphi]=\big(\Id+[d_\theta F]\Om\big)\big|_{\theta=\Om_{\tt p}(x)}$.
This together with \eqref{eq:QD^{l,s}} then implies that
\[
\det\Big(\Id+[d_\theta F]\Om\Big)\big|_{\theta=\Om_{\tt p}(x)}>\varepsilon_0>0.
\]
Since the image of the map $\Om_{\tt p} : \R^n\to\T^M$ is dense in $\T^M$ and
since $F\subseteq C^1(\T^M,\R^n)$, we conclude that for any $\theta\in\T^M$,
\begin{equation}\label{eq:det_positive}
\det\Big(\Id+[d_\theta F]\Om\Big)\ge\varepsilon_0>0.
\end{equation}
On the other side, for any $\theta\in\T^M$,
\[
[d_\theta\Phi]=\Id_{M\times M}+\Om[d_\theta F].
\]
By combining this, \eqref{eq:det_positive}, and Sylvester's determinant identity we conclude that for any $\theta\in\T^M$,
\[
\det[d_\theta\Phi]=\det\Big(\Id_{M\times M}+\Om[d_\theta F]\Big)=\det\Big(\Id+[d_\theta F]\Om\Big)\ge\varepsilon_0>0.
\]
Hence, $\Phi : \T^M\to\T^M$ is a local $C^1$-diffeomorphism.

Let us now prove that $\Phi : \T^M\to\T^M$ is onto.
Since $\T^M$ is compact we see that the image $\Phi(\T^M)$ of $\Phi$ is closed in $\T^M$. 
On the other side, since $\Phi : \T^M\to\T^M$ is a local diffeomorphism we conclude that $\Phi(\T^M)$ is open.
The connectedness of $\T^M$ then implies that $\Phi(\T^M)=\T^M$. Hence, $\Phi : \T^M\to\T^M$ is a covering map.

We will now prove that the degree of this covering map $\Phi : \T^M\to\T^M$ is one.
To this end, we will find $p\in\T^M$ such that the pre-image $\Phi^{-1}(p)$ consists of a single point.
For any $x\in\R^n$ we have
\begin{eqnarray}
\Phi\big(\Om_{\tt p}(x)\big)&=&\Om_{\tt p}(x)+\Om_{\tt p} F\big(\Om_{\tt p}(x)\big)\nonumber\\
&=&\Om_{\tt p}\Big(x+F\big(\Om_{\tt p}(x)\big)\Big)\nonumber\\
&=&\Om_{\tt p}\big(\varphi(x)\big).\label{eq:diagram_formula}
\end{eqnarray}
Hence, $\Phi$ satisfies the commutative diagram \eqref{eq:diagram}.

Consider the following discrete set in $\R^n$,
\begin{equation}\label{eq:Gamma_Om}
\Gamma_\Om:=\big\{\gamma\in\R^n\,\big|\,\Om(\gamma)\in\Z^M\big\}.
\end{equation}
The set $\Gamma_\Om$ is a discrete lattice in $\R^n$ whose rank $r$ is less than or equal to $n$ (see e.g. \cite[Lemma 3,\,\S 49]{Arnold}).
(In view of the non-resonance condition {\rm (NC)}, $r=n$ only if $n=M$.)
In particular,
\[
\Gamma_\Om:=\big\{m_1\gamma_1+...+m_r\gamma_r\,\big|\,m_1,...,m_r\in\Z\big\}
\]
for some basis $\gamma_1,...,\gamma_r\in\R^n$ of $\Gamma_\Om$.
For any $x\in\R^n$ and for any $\gamma\in\Gamma_\Om$ we have 
\begin{eqnarray}
\varphi(x+\gamma)&=&(x+\gamma)+F\big(\Om_{\tt p}(x+\gamma)\big)\nonumber\\
&=&x+F\big(\Om_{\tt p}(x)\big)+\gamma\nonumber\\
&=&\varphi(x)+\gamma,\label{eq:Gamma_ivariant}
\end{eqnarray}
where we used that $\Om(\gamma)\in\Z^M$ and $F$ is a function on $\T^M$.
This implies that for any $\gamma\in\Gamma_\Om$ we have the following commutative diagram
\begin{equation}\label{eq:translation_invariance}
\begin{tikzcd}
\R^n\arrow{r}{\varphi}&\R^n\\
\R^n\arrow{u}{\tau_\gamma}\arrow{r}{\varphi}&\R^n\arrow[swap]{u}{\tau_\gamma}
\end{tikzcd}
\end{equation}
where $\tau_\gamma(x)=x+\gamma$.
Since $\varphi : \R^n\to\R^n$ is a $C^1$-diffeomorphism \eqref{eq:translation_invariance} implies that 
for any $x\in\R^n$ and for any $\gamma\in\Gamma_\Om$ we have 
\begin{equation}\label{eq:Gamma_ivariant'}
\varphi^{-1}(x+\gamma)=\varphi^{-1}(x)+\gamma,
\end{equation}
where $\varphi^{-1} : \R^n\to\R^n$ is the inverse of the diffeomorphism $\varphi$.
Now, take $p\in\T^M$ such that $p\in\Om_{\tt p}(\R^n)$ and $\Phi(q_1)=\Phi(q_2)=p$.
Note that the assumption $\Phi(q)\in\Om_{\tt p}(\R^n)$ for some $q\in\T^M$ and \eqref{eq:Phi} imply 
$\Phi(q)=q+\Om_{\tt p}F(q)\in\Om_{\tt p}(\R^n)$ 
that gives $q\in\Om_{\tt p}(\R^n)$. Therefore, there exist ${\tilde q}_1$ and ${\tilde q}_2$ in $\R^n$ such that 
\[
\Om_{\tt p}^{-1}(q_1)={\tilde q}_1+\Gamma_\Om\quad\text{and}\quad\Om_{\tt p}^{-1}(q_2)={\tilde q}_2+\Gamma_\Om.
\]
Then we obtain from the commutative diagram \eqref{eq:diagram} and \eqref{eq:Gamma_ivariant} that 
\[
\Om_{\tt p}\big(\varphi({\tilde q}_j)+\Gamma_\Om\big)=\Om_{\tt p}\big(\varphi({\tilde q}_j+\Gamma_\Om)\big)=\Phi(q_j)=p,
\quad j=1,2.
\]
This implies that $\varphi({\tilde q}_2)=\varphi({\tilde q}_1)+\gamma$ for some $\gamma\in\Gamma_\Om$. 
Then \eqref{eq:Gamma_ivariant'} gives ${\tilde q}_2={\tilde q}_1+\gamma$ which implies that
\[
q_2=\Om_{\tt p}({\tilde q}_2)=\Om_{\tt p}({\tilde q}_1+\gamma)=\Om_{\tt p}({\tilde q}_1)=q_1.
\]
Hence the preimage $\Phi^{-1}(p)$ contains only one point. 
This implies that the degree of the covering map $\Phi : \T^M\to\T^M$ is one.
Hence, $\Phi : \T^M\to\T^M$ is a $C^1$-diffeomorphism.
Finally, since $F\in H^s(\T^M,\R^n)$ we conclude from \eqref{eq:Phi} that $\Phi\in H^s(\T^M,\T^M)$.
This completes the proof of Lemma \ref{lem:Phi-diffeomorphism}.
\end{proof}

Now we are ready to prove Proposition \ref{prop:composition}.

\begin{proof}[Proof of Proposition \ref{prop:composition}]
We will prove Proposition \ref{prop:composition} by induction in $l\ge 0$: 
First, assume that $l=0$ and note that by \eqref{eq:main1}, for any $g\in Q^s_\Om(\R^n)$ and $\varphi\in QD^s_\Om(\R^n)$,
\begin{equation}\label{eq:main1'}
g\circ\varphi=\Om_{\tt p}^*\big(G\circ\Phi),
\end{equation}
where $\Phi$ is given by \eqref{eq:Phi}, $g=\Om_{\tt p}^*G$, $\varphi=\id_{\R^n}+\Om_{\tt p}^*F$, and 
$G\in H^s(\T^M,\R)$ and $F\in H^s(\T^M,\R^n)$. (Here $\id_{\R^n}$ denotes the identity map in $\R^n$.)
By Lemma \ref{lem:pull-back}, the map
\[
Q^s_\Om(\R^n)\to H^s(\T^M,\R^n),\quad g\mapsto G,
\]
is a linear isomorphism.
In addition, $\Phi=h(\varphi)$ where $h$ is the continuous map \eqref{eq:homomorphism}.
On the other side, by Theorem 1.2 in \cite{IKT} (cf. also \cite{EM}), the composition
\[
H^s(\T^M,\R)\times D^s(\R^M)\to H^s(\T^M,\R), \quad(G,\Phi)\to G\circ\Phi,
\]
is continuous. This together with \eqref{eq:main1'} and Lemma \ref{lem:pull-back} then imply that the statement of 
Proposition \ref{prop:composition} holds when $l=0$.

Further, assume that $l\ge 1$ and let the map
\begin{equation}\label{eq:induction_hypothesis}
Q^{l-1,s}_\Om(\R^n)\times QD^{l-1,s}_\Om(\R^n)\to Q^{l-1,s}_\Om(\R^n),\quad (g,\varphi)\mapsto g\circ\varphi,
\end{equation}
be continuous. Take $g\in Q^{l,s}_\Om(\R^n)$ and $\varphi\in Q^{l,s}_\Om(\R^n)$. 
Since $Q^{l,s}_\Om(\R^n)\subseteq C^1_b(\R^n)$ we have that for any $x\in\R^n$,
\begin{equation}\label{eq:chain_rule}
\big[d_x(g\circ\varphi)\big]=[d_yg]\big|_{y=\varphi(x)}\cdot[d_x\varphi].
\end{equation}
Then \eqref{eq:chain_rule}, the induction hypotesis \eqref{eq:induction_hypothesis}, and the Banach algebra property of
$Q^{l-1,s}_\Om(\R^n)$ with $l\ge 1$ imply that the map
\[
Q^{l,s}_\Om(\R^n)\times QD^{l,s}_\Om(\R^n)\to Q^{l-1,s}_\Om(\R^n)\otimes\text{\rm Mat}_{n\times n}(\R),
\quad(g,\varphi)\mapsto\big[d(g\circ\varphi)\big],
\]
where $\text{\rm Mat}_{n\times n}(\R)$ denotes the space of $n\times n$-matrices, is continuous.
This, in view of the definition \eqref{eq:Q^{l,s}} (and \eqref{eq:Q^{l,s}-norm}) of $Q^{l,s}_\Om(\R^n)$, concludes the proof of 
Proposition \ref{prop:composition}.
\end{proof}

Further, we prove

\begin{Prop}\label{prop:the_inverse_map}
Assume that $l\in\Z_{\ge 0}$ and $s>\frac{M}{2}+1$. Then the map
\begin{equation}\label{eq:the_inverse_map}
QD^{l,s}_\Om(\R^n)\to QD^{l,s}_\Om(\R^n),\quad\varphi\mapsto\varphi^{-1},
\end{equation}
is well-defined and continuous.
\end{Prop}

Proposition \ref{prop:composition} and Proposition \ref{prop:the_inverse_map} imply

\begin{Th}\label{th:topological_group}
Assume that $l\in\Z_{\ge 0}$ and $s>\frac{M}{2}+1$.
Then $QD^{l,s}_\Om(\R^n)$ is a topological group with respect to the composition of maps.
\end{Th}

\begin{proof}[Proof of Theorem \ref{th:topological_group}]
By Proposition \ref{prop:composition} for any $\varphi,\psi\in QD^{l,s}_\Om(\R^n)$
the composition $\psi\circ\varphi\in Q^{l,s}_\Om(\R^n,\R^n)$. In view of the the definition 
\eqref{eq:QD^{l,s}} of $QD^{l,s}_\Om(\R^n)$ it follows that there exists $\varepsilon_0>0$ such that
\[
\det\big([d(\psi\circ\varphi)]\big)=
\det\big([d_y\psi]\big|_{y=\varphi(x)}\big)\det\big([d_x\varphi]\big)>\varepsilon_0>0\quad\forall x\in\R^n.
\]
This shows that $\psi\circ\varphi\in QD^{l,s}_\Om(\R^n)$.
The continuity of the composition and the inverse map \eqref{eq:the_inverse_map} then follow from 
Proposition \ref{prop:composition} and Proposition \ref{prop:the_inverse_map}.
\end{proof}

Let us also record the following 

\begin{Prop}\label{prop:homomorphism}
Assume that $s>\frac{M}{2}+1$. Then we have:
\begin{itemize}
\item[(i)] The map $h : QD^s_\Om(\R^n)\to D^s(\T^M)$, $\varphi\mapsto\Phi$, where $\Phi$ is given by \eqref{eq:Phi}, is
a homomorphism of topological groups.
\item[(ii)] The kernel of this homomorphism consists of all translations $\tau_\gamma : \R^n\to \R^n$,
$x\mapsto x+\gamma$, where $\gamma\in\Gamma_\Om$ (see \eqref{eq:Gamma_Om}).
\end{itemize}
\end{Prop}

\begin{proof}[Proof of Proposition \ref{prop:homomorphism}]
Let is first prove item $(i)$. Take $\psi,\varphi\in QD^{l,s}_\Om(\R^n)$.
By attaching to each other the two copies of the diagram \eqref{eq:diagram} corresponding respectively to $\psi$ and $\varphi$ 
we obtain the commutative diagram
\begin{equation}\label{eq:diagram_adjoint}
\begin{tikzcd}
\T^M\arrow{r}{\Psi}&\T^M\arrow{r}{\Phi}&\T^M\\
\R^n\arrow{u}{\Om_{\tt p}}\arrow{r}{\psi}&\R^n\arrow[swap]{u}{\Om_{\tt p}}\arrow{r}{\varphi}&\R^n\arrow{u}{\Om_{\tt p}}
\end{tikzcd}
.
\end{equation}
By Theorem \ref{th:topological_group}, $\psi\circ\varphi\in  QD^{l,s}_\Om(\R^n)$. In view of the uniqueness statement in 
Remark \ref{rem:uniqueness} we then conclude from \eqref{eq:diagram_adjoint} and Lemma \ref{lem:Phi-diffeomorphism} 
that $h(\psi\circ\varphi)=\Psi\circ\Phi$. This completes the proof of item $(i)$. Item $(ii)$ follows from \eqref{eq:Phi}.
\end{proof}

\begin{proof}[Proof of Proposition \ref{prop:the_inverse_map}]
Take $\varphi,\psi\in QD^s_\Om(\R^n)$.
Then, $\varphi(x)=x+f(x)$ where $f(x)=F\big(\Om_{\tt p}(x)\big)$ and 
$\psi(x)=x+g(x)$ where $g(x)=G\big(\Om_{\tt p}(x)\big)$ with $G,F\in H^s(\R^M,\R^n)$. 
We have
\begin{eqnarray}
(\psi\circ\varphi)(x)&=&\varphi(x)+g\big(\varphi(x)\big)\nonumber\\
&=&\big(x+F\big(\Om_{\tt p}(x)\big)\big)+G\Big(\Om_{\tt p}\big(x+F\big(\Om_{\tt p}(x)\big)\big)\Big)\nonumber\\
&=&x+\big\{ F+G\circ\Phi\big\}\big(\Om_{\tt p}(x)\big)\label{eq:G*F}.
\end{eqnarray}
This implies that the composition $\psi\circ\varphi\in QD^s_\Om(\R^n)$ corresponds to 
the periodic map $F+G\circ\Phi\in H^s(\T^M,\R^n)$ where $\Phi=h(\varphi)\in D^s(\T^M)$.
Hence, by taking 
\begin{equation}\label{eq:G-inverse}
G:=-F\circ\Phi^{-1} 
\end{equation}
we see from \eqref{eq:G*F} that
\begin{equation}\label{eq:phi-inverse}
\varphi^{-1}=\id_{\R^n}+\Om_{\tt p}^*\big(G\big).
\end{equation}
Moreover, it follows from \eqref{eq:G-inverse}, Theorem 1.2 in \cite{IKT}, and Lemma \ref{lem:Phi-diffeomorphism}, 
that $G\in H^s(\T^M,\R^n)$. In addition, it follows easily from \eqref{eq:phi-inverse} that there exists $\varepsilon_0>0$ such that
$\det[d(\varphi^{-1})]>\varepsilon_0>0$. This implies that $\varphi^{-1}\in QD^{l,s}_\Om(\R^n)$.
Since $F\in H^s(\T^M,\R^n)$ coordinatizes the diffeomorphism $\varphi\in QD^s_\Om(\R^n)$ 
(see Lemma \ref{lem:pull-back}) and since $\Phi\in D^s(\T^M)$ where $h$ is the continuous map 
\eqref{eq:homomorphism}, we conclude from \cite[Theorem 1.2 ]{IKT}, \eqref{eq:G-inverse}, and \eqref{eq:phi-inverse},
that the map $QD^s_\Om(\R^n)\to QD^s_\Om(\R^n)$, $\varphi\mapsto\varphi^{-1}$, is well-defined and continuous.

Now, assume that $l\ge 1$ and that the map
\begin{equation}\label{eq:induction_hypothesis2}
QD^{l-1,s}_\Om(\R^n)\to QD^{l-1,s}_\Om(\R^n),\quad\varphi\mapsto\varphi^{-1},
\end{equation}
is well-defined and continuous. For $\varphi\in QD^{l,s}_\Om(\R^n)$ we have, in view of the inclusion 
$Q^{l,s}_\Om(\R^n)\subseteq C^1_b(\R^n)$, that for any $y\in\R^n$,
\begin{equation}\label{eq:d_phi_inverse}
\big[d_y(\varphi^{-1})\big]=\big[[d_x\varphi]|_{x=\varphi^{-1}(y)}\big]^{-1}.
\end{equation}
This together with induction hypothesis \eqref{eq:induction_hypothesis2}, Proposition \ref{prop:composition},
the Banach algebra properties of $Q^{l-1,s}_\Om(\R^n)$, and Lemma \ref{lem:division} below, implies that
the map
\[
QD^{l,s}_\Om(\R^n)\to Q^{l-1,s}_\Om(\R^n)\otimes\text{\rm Mat}_{n\times n}(\R),\quad\varphi\mapsto\big[d(\varphi^{-1})\big],
\]
is well-defined and continuous. This completes the proof of Proposition \ref{prop:the_inverse_map}.
\end{proof}

The following lemma is used in the proof of Proposition \ref{prop:the_inverse_map}.
\begin{Lem}\label{lem:division}
Assume that $l\in\Z_{\ge 0}$, $s>\frac{M}{2}$, and let $f_0\in Q^{l,s}_\Om(\R^n)$ be such that
$|f_0(x)|>\varepsilon_0>0$ for any $x\in\R^n$. Then $1/f_0\in Q^{l,s}_\Om(\R^n)$. Moreover,
there exists an open neighborhood $U$ of zero in $Q^{l,s}_\Om(\R^n)$ such that the map
\[
U\to Q^{l,s}_\Om(\R^n),\quad f\mapsto 1/(f_0+f),
\]
is continuous.\footnote{In fact, this is a real analytic map.}
\end{Lem}
\noindent Lemma \ref{lem:pull-back} and an induction argument in $l\ge 0$ reduces the proof of Lemma \ref{lem:division} to 
an analogous statement with $Q^{l,s}_\Om(\R^n)$ replaced by $H^s(\T^M,\R)$. The case of the Sobolev space $H^s$ follows from 
Lemma B.2 in \cite{IKT}.

In the remaining part of this Section we will prove that the composition and the inverse in the topological group $QD^{l,s}_\Om(\R^n)$ 
enjoy additional regularity properties. These properties are similar to the ones of the groups $D^s(\R^n)$ and $D^s(X)$ where $X$ is 
a smooth compact manifold without boundary (see e.g. \cite[Theorem 1.1 and Theorem 1.2]{IKT}). More specifically, one has

\begin{Th}\label{th:composition} 
Assume that $l\in\Z_{\ge 0}$ and $s>\frac{M}{2}+1$. Then for any $r\in\Z_{\ge 0}$ the maps
\begin{equation}\label{eq:composition_regularity}
Q^{l+r,s}_\Om(\R^n)\times QD^{l,s}_\Om(\R^n)\to Q^{l,s}_\Om(\R^n),\quad(g,\varphi)\mapsto g\circ\varphi,
\end{equation}
and
\begin{equation}\label{eq:inverse_regularity}
QD^{l+r,s}_\Om(\R^n)\to QD^{l,s}_\Om(\R^n),\quad\varphi\mapsto\varphi^{-1},
\end{equation}
are $C^r$-smooth.
\end{Th}

We first prove

\begin{Lem}\label{lem:almost_lipschitz}
Assume that $l\in\Z_{\ge 0}$ and $s>\frac{M}{2}+1$.
For any $\varphi_0\in QD^{l,s}_\Om(\R^n)$ there exist an open neighborhood $U(\varphi_0)$ of $\varphi_0$ in
$QD^{l,s}_\Om(\R^n)$ and a constant $C>0$ such that for any $g\in Q^{l,s}_\Om(\R^n)$ and for any 
$\varphi\in U(\varphi_0)$ one has
\[
\|g\circ\varphi\|_{l,s}\le C\|g\|_{l,s}.
\]
\end{Lem}

\begin{proof}[Proof of Lemma \ref{lem:almost_lipschitz}]
First, recall that we identify $QD^{l,s}_\Om(\R^n)$ with an open set in $Q^{l,s}_\Om(\R^n)$ -- see 
\eqref{eq:QD^{l,s}} and Lemma \ref{lem:differentiable_structure}.
The statement of Lemma \ref{lem:almost_lipschitz} follows directly from Proposition \ref{prop:composition}. 
In fact, since $g\circ\varphi\to 0$ in $Q^{l,s}_\Om(\R^n)$ as $(g,\varphi)\to (0,\varphi_0)$ in 
$Q^{l,s}_\Om(\R^n)\times QD^{l,s}_\Om(\R^n)$ we conclude from  Proposition \ref{prop:composition} that there exist 
$C_1>0$ and $\varepsilon>0$ such that
\[
\|g\circ\varphi\|_{l,s}\le C_1
\]
for any $g\in Q^{l,s}_\Om(\R^n)$ with $\|g\|_{l,s}\le\varepsilon$ and for any 
$\varphi\in U(\varphi_0)\equiv\id_{\R^n}+B(f_0)$ where $B(f_0)$ is an open ball centered at 
$f_0\equiv\varphi_0-\id_{\R^n}$ in $Q^{l,s}_\Om(\R^n)$ and such that 
$U(\varphi_0)\subseteq QD^{l,s}_\Om(\R^n)$. 
This and the linearity of the composition with respect to the first argument then imply that
\[
\|g\circ\varphi\|_{l,s}\le\left(\frac{C_1}{\varepsilon}\right)\,\|g\|_{l,s}
\]
for any $g\in Q^{l,s}_\Om(\R^n)$ and for any $\varphi\in U(\varphi_0)$.
Finally, by setting $C=C_1/\varepsilon$ we complete the proof of the Lemma.
\end{proof}

In fact, we have also the following variant of Lemma \ref{lem:almost_lipschitz}.

\begin{Lem}\label{lem:almost_bilinear}
Assume that $l\in\Z_{\ge 0}$ and $s>\frac{M}{2}+1$.
For any $\varphi_0\in QD^{l,s}_\Om(\R^n)$ there exist an open neighborhood $U(\varphi_0)$ of $\varphi_0$ in
$QD^{l,s}_\Om(\R^n)$ and a constant $C>0$ such that for any $g\in Q^{l+1,s}_\Om(\R^n)$ and for any 
$\varphi\in U(\varphi_0)$ one has
\[
\|g\circ\varphi-g\circ\varphi_0\|_{l,s}\le C\|g\|_{l+1,s}\|\varphi-\varphi_0\|_{l,s}.
\]
\end{Lem}

\begin{proof}[Proof of Lemma \ref{lem:almost_bilinear}]
In view of the inclusion $Q^{l,s}_\Om(\R^n)\subseteq C^1_b(\R^n)$ we have that for any 
$\varphi\in U(\varphi_0)\equiv\id_{\R^n}+B(f_0)$, 
where $B(f_0)$ is a given open ball centered at $f_0\equiv\varphi_0-\id_{\R^n}$ in $Q^{l,s}_\Om(\R^n)$ and such that 
$U(\varphi_0)\subseteq QD^{l,s}_\Om(\R^n)$, and for any $x\in\R^n$,
\begin{eqnarray}
g\big(\varphi(x)\big)-g\big(\varphi_0(x)\big)&=&\int_0^1[d_yg]|_{y=\varphi_0(x)+t\delta\varphi(x)}\cdot\delta\varphi(x)\,dt\nonumber\\
&=&\Big(\int_0^1[d_yg]|_{y=\varphi_0(x)+t\delta\varphi(x)}\,dt\Big)\cdot\delta\varphi(x)\label{eq:mean_value}
\end{eqnarray}
where $\delta\varphi:=\varphi-\varphi_0\in Q^{l,s}_\Om(\R^n)$.
Since, $g\in Q^{l+1,s}_\Om(\R^n)$ and $\varphi,\varphi_0\in Q^{l,s}_\Om(\R^n)$ we conclude from Proposition \ref{prop:Q^{l,s}}
and Proposition \ref{prop:composition} that the curve
\[
[0,1]\to Q^{l,s}_\Om(\R^n)\otimes\text{\rm Mat}_{n\times n}(\R),\quad t\mapsto [d_yg]|_{y=\varphi_0(x)+t\delta\varphi(x)},
\]
is continuous. This implies that that the integral in \eqref{eq:mean_value} has convergent Riemann sums in $Q^{l,s}_\Om(\R^n)$ and
that \eqref{eq:mean_value} holds in $Q^{l,s}_\Om(\R^n)$. This and the Banach algebra property of $Q^{l,s}_\Om(\R^n)$ then imply that
\begin{equation}\label{eq:formula2}
\|g\circ\varphi-g\circ\varphi_0\|_{l,s}\le C_1 \sup_{\varphi\in U(\varphi_0)}\big\|[dg]\circ\varphi\big\|_{l,s}\|\delta\varphi\|_{l,s}
\end{equation}
for some constants $C_1>0$ that is independent of the choice of $g\in Q^{l+1,s}_\Om(\R^n)$ and $\varphi\in U(\varphi_0)$,
and the neighborhood $U(\varphi_0)$ in $QD^{l,s}_\Om(\R^n)$ is chosen such that $\big\|[dg]\circ\varphi\big\|_{l,s}$ is bounded 
uniformly in $U(\varphi_0)$ (see Proposition \ref{prop:composition}, Lemma \ref{lem:almost_lipschitz}).
Finally, the Lemma follow from \eqref{eq:formula2} and lemma \ref{lem:almost_lipschitz}.
\end{proof}

Now, we are ready to prove Theorem \ref{th:composition}.

\begin{proof}[Proof of Theorem \ref{th:composition}]
For the proof of the regularity of the composition \eqref{eq:composition_regularity} we follow the lines of the proof of \cite[Proposition 2.9]{IKT}. 
If $r=0$ the statement of Theorem \ref{th:composition} follows directly from
Proposition \ref{prop:composition} and Proposition \ref{prop:the_inverse_map}. Further, assume that $r\ge 1$ and take 
$g,\delta g\in Q^{l+r,s}_\Om(\R^n)$, $\varphi\in QD^{l,s}_\Om(\R^n)$, and $\delta\varphi\in Q^{l,s}_\Om$ such that 
$\varphi+\delta\varphi\in U(\varphi)\equiv\id_{\R^n}+B(f)$ where $B(f)$ is a given open ball centered at $f\equiv\varphi-\id_{\R^n}$ in 
$Q^{l,s}_\Om$ and such that $U(\varphi)\subseteq QD^{l,s}_\Om(\R^n)$. 
In view of the inclusion $Q^{l+r,s}_\Om(\R^n)\subseteq C^{l+r+1}_b(\R^n)\subseteq C^{r}_b(\R^n)$ 
(see Proposition \ref{prop:Q^{l,s}} $(iii)$) one sees from Taylor's formula with remainder in integral form that for any $x\in\R^n$,
\begin{equation}\label{eq:taylor_expansion1}
g\big(\varphi(x)+\delta\varphi(x)\big)=\sum_{|\beta|\le r}\frac{1}{\beta!}\big(\p_x^\beta g\big)\big(\varphi(x)\big){\big(\delta\varphi(x)\big)^\beta}+
\mathcal{R}_1\big(g,\varphi,\delta\varphi\big)(x)
\end{equation}
where the remainder $\mathcal{R}_1\big(g,\varphi,\delta\varphi\big)(x)$ is given by
\begin{equation}\label{eq:remainder1}
\sum_{|\beta|=r}\frac{r}{\beta!}\int_0^1(1-t)^{r-1}
\Big(\big(\p_x^\beta g\big)\big(\varphi(x)+t\delta\varphi(x)\big)-\big(\p_x^\beta g\big)\big(\varphi(x)\big)\Big)\,dt
\cdot\big(\delta\varphi(x)\big)^\beta.
\end{equation}
Similarly, for any $\varphi$, $\delta\varphi$, and $\delta g$ as above and for any $x\in\R^n$ we have
\begin{equation}\label{eq:taylor_expansion2}
\delta g\big(\varphi(x)+\delta\varphi(x)\big)=\sum_{|\beta|\le r-1}\frac{1}{\beta!}
\big(\p_x^\beta \delta g\big)\big(\varphi(x)\big){\big(\delta\varphi(x)\big)^\beta}+
\mathcal{R}_2\big(\varphi,\delta g,\delta\varphi\big)(x)
\end{equation}
where $\mathcal{R}_2\big(\varphi,\delta g,\delta\varphi\big)(x)$ is given by
\begin{equation}\label{eq:remainder2}
\sum_{|\beta|=r}\frac{r}{\beta!}\int_0^1(1-t)^{r-1}
\big(\p_x^\beta\delta g\big)\big(\varphi(x)+t\delta\varphi(x)\big)\,dt
\cdot\big(\delta\varphi(x)\big)^\beta.
\end{equation}
It follows from Proposition \ref{prop:composition}, Proposition \ref{prop:the_inverse_map}, and the Banach algebra property of 
$Q^{l,s}_\Om(\R^n)$ that the integrals in \eqref{eq:remainder1} and \eqref{eq:remainder2} have convergent Riemann sums in 
$Q^{l,s}_\Om(\R^n)$ and the equalities \eqref{eq:taylor_expansion1} and \eqref{eq:taylor_expansion2} also hold in $Q^{l,s}_\Om(\R^n)$.
Summing up these two equalities we see that for any $g,\delta g\in Q^{l+r,s}_\Om(\R^n)$, $\varphi\in QD^{l,s}_\Om(\R^n)$, and for any
$\delta\varphi\in Q^{l,s}_\Om$ as above,
\begin{equation}\label{eq:expansion3}
(g+\delta g)\circ(\varphi+\delta\varphi)=g\circ\varphi+\sum_{k=1}^r \frac{1}{k!}P_k\big(g,\varphi\big)(\delta g,\delta\varphi)+
\mathcal{R}\big(g,\varphi,\delta g,\delta\varphi\big)
\end{equation}
where
\[
\frac{1}{k!}\,P_k\big(g,\varphi\big)(\delta g,\delta\varphi):=
\sum_{|\beta|=k}\frac{1}{\beta!}\big(\p_x^\beta g\big)\circ\varphi\cdot(\delta\varphi)^\beta+
\sum_{|\beta|=k-1}\frac{1}{\beta!}\big(\p_x^\beta \delta g\big)\circ\varphi\cdot(\delta\varphi)^\beta
\]
and
\begin{eqnarray}
\mathcal{R}\big(g,\varphi,\delta g,\delta\varphi\big)&:=&\mathcal{R}_1(g,\varphi,\delta\varphi\big)+
\mathcal{R}_2(\varphi,\delta\varphi,\delta g\big)\nonumber\\
&=&\sum_{|\beta|=r}\rho_{\beta}(g,\varphi,\delta g,\delta\varphi)\cdot(\delta\varphi)^\beta\label{eq:remainder3}
\end{eqnarray}
with
\begin{eqnarray}
\rho_{\beta}(g,\varphi,\delta g,\delta\varphi)&:=&\frac{r}{\beta!}\int_0^1(1-t)^{r-1}
\Big(\big(\p_x^\beta g\big)\circ\big(\varphi+t\delta\varphi\big)-\big(\p_x^\beta g\big)\circ\varphi(x)\Big)\,dt\nonumber\\
&+&\frac{r}{\beta!}\int_0^1(1-t)^{r-1}
\big(\p_x^\beta\delta g\big)\circ\big(\varphi+t\delta\varphi\big)\,dt\label{eq:remainder3'}
\end{eqnarray}
for any multi-index $\beta\in\Z_{\ge 0}^n$ with $|\beta|=r$.
In view of the Banach algebra property of $Q^{l,s}_\Om(\R^n)$ for any $g\in Q^{l+r,s}_\Om(\R^n)$ and 
$\varphi\in QD^{l,s}_\Om(\R^n)$,
\[
P_k(g,\varphi)\in\mathcal{P}^k\Big(Q^{l+r,s}_\Om\times Q^{l,s}_\Om,Q^{l,s}_\Om\Big),\quad 1\le k\le r,
\]
where $\mathcal{P}^k(X,Y)$ denotes the Banach space of polynomial maps of degree $k$ from a normed
space $X$ to a Banach space $Y$ supplied with the uniform norm.
Moreover, using again the Banach algebra property of $Q^{l,s}_\Om(\R^n)$, Proposition \ref{prop:composition},
and Lemma \ref{lem:almost_bilinear} one easily sees that for any $1\le k\le r$ the map
\[
P_k : Q^{l+r,s}_\Om(\R^n)\times QD^{l,s}_\Om(\R^n)\to\mathcal{P}^k
\Big(Q^{l+r,s}_\Om\times Q^{l,s}_\Om,Q^{l,s}_\Om\Big)
\]
is continuous. 
For any $\varphi\in QD^{l,s}_\Om(\R^n)$ denote by $B_\bullet(f)$ the ball centered at $f\equiv\varphi-\id_{\R^n}$ in 
$Q^{l,s}_\Om(\R^n)$ of maximal radius such that $U_\bullet(\varphi)\equiv\id_{\R^n}+B_\bullet(f)\subseteq QD^{l,s}_\Om(\R^n)$.
Consider the open set $V$ of elements $(g,\varphi,\delta g,\delta\varphi)$ in 
$Q^{l+r,s}_\Om\times QD^{l,s}_\Om(\R^n)\times Q^{l+r,s}_\Om\times Q^{l,s}_\Om$,
\[
V:=\Big\{(g,\delta g)\in Q^{l+r,s}_\Om\times Q^{l+r,s}_\Om,\varphi\in QD^{l,s}_\Om,\delta\varphi\in Q^{l,s}_\Om\,\Big|\,
\varphi+\delta\varphi\in U_\bullet(\varphi)\Big\}.
\]
It follows from \eqref{eq:remainder3'} and Proposition \ref{prop:composition} that 
for any $\beta$ with $|\beta|=r$ the map $\rho_\beta : V\to Q^{l,s}_\Om(\R^n)$ is continuous. 
This together with \eqref{eq:remainder3}, \eqref{eq:remainder3'}, and the Banach algebra property of $Q^{l,s}_\Om(\R^n)$ then 
implies that the map
\[
\widetilde{\mathcal{R}} : V\to \mathcal{P}^k\Big(Q^{l+r,s}_\Om\times Q^{l,s}_\Om,Q^{l,s}_\Om\Big),\quad
(g,\varphi,\delta g,\delta\varphi)\mapsto\big[w\mapsto\widetilde{\mathcal{R}}(g,\varphi,\delta g,\delta\varphi;w)\big],
\]
where
\[
\widetilde{\mathcal{R}}(g,\varphi,\delta g,\delta\varphi;w):=\sum_{|\beta|=r}\rho_\beta(g,\varphi,\delta g,\delta\varphi)\cdot w^\beta,
\quad w\in Q^{l,s}_\Om(\R^n),
\]
is continuous. Note that
\begin{equation}\label{eq:R=tildeR}
\mathcal{R}\big(g,\varphi,\delta g,\delta\varphi\big)=\widetilde{\mathcal{R}}(g,\varphi,\delta g,\delta\varphi;\delta\varphi).
\end{equation}
Since $\widetilde{\mathcal{R}}\big(g,\varphi,0,0;\cdot\big)=0$ we then conclude from \eqref{eq:expansion3}, \eqref{eq:R=tildeR}, and 
the converse to Taylor's theorem (see e.g. \cite[Theorem 2.4.15]{AMR}) that \eqref{eq:composition_regularity} is a $C^r$-map.

The regularity of the inverse map \eqref{eq:inverse_regularity} now follows easily from the implicit function theorem -- 
see the proof of \cite[Proposition 2.13]{IKT}. In fact, if $r=0$ then the statement follows from Proposition \ref{prop:the_inverse_map}.
Now, assume that $r\ge 1$, take $\varphi_0\in QD^{l+r,s}_\Om(\R^n)$, and denote by $\psi_0$ its inverse 
$\psi_0=\varphi_0^{-1}\in QD^{l+r,s}_\Om(\R^n)$. It follows from the regularity of \eqref{eq:composition_regularity} that
\[
F : QD^{l+r,s}_\Om(\R^n)\times QD^{l,s}_\Om(\R^n)\to Q^{l,s}_\Om(\R^n),\quad(\varphi,\psi)\mapsto\varphi\circ\psi,
\]
is $C^r$-smooth. The (functional) partial derivarive of $F$ with respect to the second argument 
$D_2F(\varphi_0,\psi_0) : Q^{l,s}_\Om(\R^n)\to Q^{l,s}_\Om(\R^n)$ is 
\[
D_2F(\varphi_0,\psi_0)(\delta\psi)=\big[[d\varphi_0]\circ\varphi_0^{-1}\big]\cdot\delta\psi.
\]
Since $r\ge 1$ it follows from the Banach algebra property of $Q^{l,s}_\Om(\R^n)$ that the components of the Jacobian matrix
$[d\varphi_0]\circ\varphi_0^{-1}$ and it's (matrix) inverse $\big[[d\varphi_0]\circ\varphi_0^{-1}\big]^{-1}$ belong to
$Q^{l,s}_\Om(\R^n)$ (see Lemma \ref{lem:division}). This implies that 
\[
D_2F(\varphi_0,\psi_0) : Q^{l,s}_\Om(\R^n)\to Q^{l,s}_\Om(\R^n)
\]
is an isomorphism. In addition, $F(\varphi_0,\psi_0)=\id_{\R^n}$.
Hence, by the implicit function theorem in Banach spaces, there exist an open neighborhood $U(\varphi_0)$ of $\varphi_0$ in 
$QD^{l+r,s}_\Om(\R^n)$ and a $C^r$-map
\begin{equation}\label{eq:Psi}
\Psi : U(\varphi_0)\to Q^{l,s}_\Om(\R^n),\quad\varphi\mapsto\Psi(\varphi),
\end{equation}
such that $F\big(\varphi,\Psi(\varphi)\big)=\id_{\R^n}$ for any $\varphi\in U(\varphi_0)$.
In particular we see that for any $\varphi\in U(\varphi_0)$, $\varphi^{-1}=\Psi(\varphi)$.
Hence, by the regularity of the map \eqref{eq:Psi}, the map 
\[
U(\varphi_0)\to QD^{l,s}_\Om(\R^n),\quad\varphi\mapsto\varphi^{-1}=\Psi(\varphi),
\] 
is a $C^r$-map. Since $\varphi_0\in QD^{l+r,s}_\Om(\R^n)$ was chosen arbitrarily we conclude that 
\eqref{eq:inverse_regularity} is a $C^r$-map.
\end{proof}

\section{Proof of Theorem \ref{th:euler}}\label{sec:euler}
Consider the Euler equation \eqref{eq:euler},
\begin{equation}\label{eq:euler1}
\left\{
\begin{array}{l}
u_t+u\cdot\nabla u=-\nabla\pp,\quad\Div u=0,\\
u|_{t=0}=u_0,
\end{array}
\right.
\end{equation}
and assume that for given $l\ge 2$, $s>\frac{M}{2}+1$, and $T>0$,
\begin{equation}\label{eq:u-solution}
u\in C\big([-T,T],Q^{l,s}_\Om\big)\cap C^1(\big[-T,T],Q^{l-1,s}_\Om\big)
\end{equation}
is a solution of \eqref{eq:euler1}. Then, by \eqref{eq:euler1} and the Banach algebra property of $Q^{l-1,s}_\Om(\R^n)$
(Proposition \ref{prop:Q^{l,s}} (ii)),
\begin{equation*}
\nabla\pp\in C\big([-T,T],Q^{l-1,s}_\Om\big)
\end{equation*}
and, by assumption, $\pp(t)\in Q^s_\Om(\R^n)$ and has mean-value zero for any $t\in[-T,T]$.
By applying $\Div$ to the both sides of \eqref{eq:euler1}, we obtain that
\begin{equation}\label{eq:pressure1}
\Div\big(u\cdot\nabla u\big)=-\Delta\pp\,.
\end{equation}
Since $\Div u=0$ we then conclude by a direct computation that
\begin{equation}\label{eq:Q-relation}
\Div\big(u\cdot\nabla u\big)=\tr\big([du]^2\big)+u\cdot\nabla\big(\Div u\big)=\tr\big([du]^2\big)
\end{equation}
where $[du]$ is the Jacobian matrix of $w$ and $\tr$ is the trace of an $n\times n$-matrix.
For $w\in Q^{l,s}_\Om$ with $l\ge 0$, $s>M/2$, denote
\begin{equation}\label{eq:Q}
Q(w):=\tr\big([dw]^2\big)
\end{equation}
and note that $Q(w)\in Q^{l-1,s}_\Om(\R^n)$ by Proposition \ref{prop:Q^{l,s}}.

Let $\Lambda_m=(\Lambda_{m,1},...,\Lambda_{m,n})$, $m\in\Z^M$, 
be the exponents appearing in the Fourier expansion \eqref{eq:f_expansion1} of 
the elements of $Q^{r,s}_\Om$ with $r\ge 0$ and $s>M/2$. Consider the subsets of indices
\begin{align}
&I_\bullet:=\big\{m\in\Z^M\,\big|\,|\Lambda_m|\le 1,\,\,1\le j\le n\big\},\label{eq:I_bullet}\\
&I_\infty:=\Z^M\setminus I_\bullet,\label{eq:I_infty}
\end{align}
and define the closed subspaces in $Q^{l,s}_\Om$,
\begin{align}
&Q^{r,s}_{\Om,\bullet}:=\big\{w\in Q^{r,s}_\Om\,\big|\,\widehat{w}_m=0\,\,\,
\text{\rm for }\,\,m\in I_\infty\big\},\label{eq:Q_bullet}\\
&Q^{r,s}_{\Om,\infty}:=\big\{w\in Q^{r,s}_\Om\,\big|\,\widehat{w}_m=0\,\,\,
\text{\rm for }\,\,m\in I_\bullet\big\}.\label{eq:Q_infty}
\end{align}
By construction,
\begin{equation}\label{eq:splitting}
Q^{r,s}_\Om=Q^{r,s}_{\Om,\bullet}\oplus_\perp Q^{r,s}_{\Om,\infty},
\end{equation}
where the spaces are orthogonal with respect to the scalar product 
\eqref{eq:Q^{l,s}-scalar_product} on $Q^{r,s}_\Om$. Let 
\begin{equation}\label{eq:projections}
\Pi_\bullet : Q^{r,s}_\Om\to Q^{r,s}_{\Om,\bullet},\quad 
\Pi_\infty : Q^{r,s}_\Om\to Q^{r,s}_{\Om,\infty},
\end{equation}
be the (orthogonal) projections corresponding to the splitting \eqref{eq:splitting}.
Clearly, $\Pi_\infty+\Pi_\bullet=\id_{Q^{r,s}_\Om}$ and the projections \eqref{eq:projections}
commute with the differentiation in the scale $Q^{r,s}_\Om$, $r\ge 0$, $s>M/1$. 
Our first observation concerns the space $Q^{r,s}_{\Om,\bullet}$.

\begin{Lem}\label{lem:Q_bullet-smooth}
For any $r,\tau\in\Z_{\ge 0}$ and $s>M/2$ we have that 
$Q^{r,s}_{\Om,\bullet}\subseteq Q^{r+\tau,s}_\Om\subseteq C^\infty_b$
and the inclusions are bounded. In particular, for any $r,\tau\ge 0$ the map
\[
\Pi_\bullet : Q^{r,s}_\Om\to Q^{r+\tau,s}_\Om
\]
is well-defined and bounded.
\end{Lem}

\begin{proof}[Proof of Lemma \ref{lem:Q_bullet-smooth}]
It is enough to consider scalar valued functions. 
Take $f\in Q^{r,s}_\Om(\R^n)$ where $r\ge 0$ and $s>M/2$. 
It then follows from Lemma \ref{lem:Q^{l,s}} (ii) that
\[
f(x)=\sum_{m\in I_\bullet} \widehat{f}_m e^{i (\Lambda_m,x)}
\]
where $\sum_{m\in I_\bullet}|\widehat{f}_m|^2\langle\Lambda_m\rangle^{2r}\m^{2s}<\infty$.
This and the definition \eqref{eq:I_bullet} of the set $I_\bullet$ then imply that 
$|\Lambda_m|$ are bounded uniformly in $m\in I_\bullet$ and hence for any $\tau\ge 0$,
\[
\sum_{m\in I_\bullet}|\widehat{f}_m|^2\langle\Lambda_m\rangle^{2(r+\tau)}\m^{2s}<\infty.
\]
Applying  Lemma \ref{lem:Q^{l,s}} (ii) one more time we see that $f\in Q^{r+\tau,s}_\Om(\R^n)$ and
the inclusions
\[
Q^{r,s}_{\Om,\bullet}(\R^n)\subseteq Q^{r+\tau,s}_\Om(\R^n)
\]
are bounded. Since this happens for any $\tau\ge 0$ we then conclude from 
Proposition \ref{prop:Q^{l,s}} (iii) that 
\[
Q^{r,s}_{\Om,\bullet}(\R^n)\subseteq\bigcap_{\tau\ge 0} Q^{\tau,s}_\Om(\R^n)\subseteq C^\infty_b(\R^n)
\]
and the inclusions are bounded.
\end{proof}

\begin{Rem}
More generally, since $f\in Q^{r,s}_{\Om,\bullet}(\R^n)\subseteq S'(\R^n)$ is a tempered distribution 
whose Fourier transform
\[
(\mathcal{F}f)(\xi)=(2\pi)^n\sum_{m\in I_\bullet}\widehat{f}_m\delta(\xi-\Lambda_m)
\]
has support $I_\bullet$ inside the centered at zero closed ball of radius $\le 1$ in $\C^n$, we conclude from 
Paley-Wiener's theorem that $f$ extends to an entire function on $\C^n$ such that 
$|f(\xi)|\le C e^{\mathop{\rm Im}(\xi)})$ (cf. e.g. \cite{Shubin2}). 
\end{Rem}

Our next observation is the following lemma.

\begin{Lem}\label{lem:Delta-isomorphism}
For any $r\ge 0$ and $s>M/2$ the map
\begin{equation*}
\Delta : Q^{r+2,s}_{\Om,\infty}\to Q^{r,s}_{\Om,\infty}
\end{equation*}
is a linear isomorphism.
\end{Lem}

\noindent The lemma follows directly from the estimate 
$\frac{1}{2}\,\langle\Lambda_m\rangle\le|\Lambda_m|\le\langle\Lambda_m\rangle,\quad m\in I_\infty$,
and Lemma \ref{lem:Q^{l,s}} (ii).

In view of \eqref{eq:pressure1} and \eqref{eq:Q-relation} we have
\begin{align*}
-\Delta\pp&=\Div\big(u\cdot\nabla u\big)=\big(\Pi_\infty+\Pi_\bullet\big)\Div\big(u\cdot\nabla u\big)\\
&=\Pi_\infty\circ Q(u)+\Pi_\bullet\Div(u\cdot\nabla u\big)\\
&=\Pi_\infty\circ Q(u)+\Div\circ\,\Pi_\bullet(u\cdot\nabla u\big).
\end{align*}
This implies that
\begin{equation}\label{eq:pressure2}
-\Delta\big(\nabla\pp\big)=\Pi_\infty\circ\nabla\circ Q(u)+\big(\nabla\circ\Div\big)\circ\,\Pi_\bullet\circ D(u),
\end{equation}
where we set 
\[
D(w):=w\cdot\nabla w
\]
for $w\in Q^{l,s}_\Om$. Note that for any $r\ge 0$ and $s>M/2$ the map
\[
\nabla\circ\Div : Q^{r+2,s}_\Om\to Q^{r,s}_\Om
\]
appearing in \eqref{eq:pressure2} is bounded and its image is contained in the image of the Laplace operator
$\Delta : Q^{r+2,s}_\Om\to Q^{r,s}_\Om$. In order to see this, define the linear map
\begin{equation}\label{eq:Delta-inverse-map}
\Delta^{-1}\circ\nabla\circ\Div : Q^{r+2,s}_\Om\to Q^{r+2,s}_\Om,\quad w\mapsto\big(\Delta^{-1}\circ\nabla\circ\Div\big)(w),
\end{equation}
by setting
\begin{equation}\label{eq:Delta-inverse}
\widehat{\Big(\Delta^{-1}\partial_j\partial_k w_k\Big)}_m:=
\left\{
\begin{array}{cc}
\frac{\Lambda_{m,j}\Lambda_{m,k}}{|\Lambda_m|^2}\widehat{(w_k)}_m,&\quad m\in\Z^M\setminus\{0\},\\
0,&m=0,
\end{array}
\right.
\end{equation}
for any $1\le k,j\le n$. The map \eqref{eq:Delta-inverse-map} is bounded since 
$\frac{\Lambda_{m,j}\Lambda_{m,k}}{|\Lambda_m|^2}\le 1$ for any $m\in\Z^M\setminus\{0\}$.
By combining  \eqref{eq:pressure2} with \eqref{eq:Delta-inverse}, Lemma \ref{lem:Delta-isomorphism},
and the fact that $\nabla\pp\in Q^{l-1,s}_\Om(\R^n)$, we conclude that
\begin{equation}\label{eq:pressure3}
-\nabla\pp(t)=\PP\big(u(t)\big)+c(t),\quad t\in[-T,T],
\end{equation}
where $c : [-T,T]\to\R^n$ is a curve in $\R^n$ and
\begin{equation}\label{eq:PP}
\quad\PP(w):=\big(\Delta^{-1}\circ\Pi_\infty\big)\circ\nabla\circ Q(w)+
\big(\Delta^{-1}\circ\nabla\circ\,\Div\big)\circ\Pi_\bullet\,\circ D(w),\quad w\in Q^{l,s}_\Om.
\end{equation}
Here one uses that the kernel of $\Delta : S'(\R^n)\to S'(\R^n)$ consists of harmonic polynomials on $\R^n$
(see e.g. \cite[\S\,5.10]{Shubin2}) and that $\nabla\pp$ and $\PP(u)$ are quasi-periodic, and hence
uniformly bounded on $\R^n $ (cf.  Proposition \ref{prop:Q^{l,s}} (iii), Lemma \ref{lem:PP} below). 
We have

\begin{Lem}\label{lem:PP}
For any $l\ge 2$ and $s>M/2$ the map
\[
\PP : Q^{l,s}_\Om\to Q^{l,s}_\Om,\quad w\mapsto\PP(w),
\]
is well-defined and real analytic. Moreover, $\PP(w)$ has mean-value zero for any $w\in Q^{l,s}_\Om$.
\end{Lem}

\begin{proof}[Proof of Lemma \ref{lem:PP}]
Assume that $l\ge 2$ and $s>M/2$.
It follows from Proposition \ref{prop:Q^{l,s}} (i) and the Banach algebra property of 
$Q^{r,s}_\Om$ (Proposition \ref{prop:Q^{l,s}} (ii)) that the polynomial map \eqref{eq:Q},
\[
Q : Q^{l,s}_\Om\to Q^{l-1}_\Om,\quad w\mapsto Q(w),
\]
is real-analytic. This implies that the map
$\Pi_\infty\circ\nabla\circ Q : Q^{l,s}_\Om\to Q^{l-2}_{\Om,\infty}$
is real analytic. By combining this with Lemma \ref{lem:Delta-isomorphism}
we conclude that 
\[
\big(\Delta^{-1}\circ\Pi_\infty\big)\circ\nabla\circ Q : Q^{l,s}_\Om\to Q^{l,s}_\Om
\]
is well-defined and real analytic. Similar arguments involving the smoothing of the projection 
operator $\Pi_\bullet$ in Lemma \ref{lem:Q_bullet-smooth}
and the boundedness of the map \eqref{eq:Delta-inverse-map} imply that the map
\[
\big(\Delta^{-1}\circ\nabla\circ\Div \big)\circ\Pi_\bullet\,\circ D : Q^{l,s}_\Om\to Q^{l,s}_\Om
\]
is well-defines and real analytic.

Let us now prove  the second statement of the lemma. Take $w\in Q^{l,s}_\Om$.
The second summand in \eqref{eq:PP} has mean-value zero by the definition \eqref{eq:Delta-inverse} of the map
\eqref{eq:Delta-inverse-map}. It follows from \eqref{eq:Q} and Lemma \ref{lem:Delta-isomorphism} 
that the first term in \eqref{eq:PP} is the gradient of the quasi-periodic function 
$\Delta^{-1}\circ\Pi_\infty\circ Q(w)\in Q^{l+1,s}_\Om(\R^n)$. 
By Lemma \ref{lem:Q^{l,s}} (ii) we then conclude that the $0$'th Fourier coefficients of
the components of the first summand vanish. This completes the proof of the lemma.
\end{proof}

On the other side, since $\Div u=0$ the $j$-th component of $u\cdot\nabla u$ can be written as
$u\cdot\nabla u_j=\sum_{k=1}^n\partial_{x^k}\big(u_ku_j\big)$ where 
$(u_1,...,u_n)$ are the components of the fluid velocity \eqref{eq:u-solution}. This implies that
\begin{equation}\label{eq:momentum_tensor_form}
\Div\big(u\cdot\nabla u\big)=\sum_{k,j=1}^n\partial_{x_j}\partial_{x_k}\big(u_ku_j\big).
\end{equation}
By applying $\Div$ to the both sides of \eqref{eq:euler1} we then obtain that
\begin{equation}
-\Delta\pp=\sum_{k,j=1}^n\partial_{x_j}\partial_{x_k}\big(u_ku_j\big).
\end{equation}
Since $u\in C\big([-T,T],Q^{l,s}_\Om\big)$ we conclude from the assumption that $\pp\in Q^s_\Om(\R^n)$ 
has mean-value zero and Proposition \ref{prop:Q^{l,s}} that
\begin{equation}\label{eq:P-regularity}
\pp=\sum_{k,j=1}^n\Delta^{-1}\partial_{x_j}\partial_{x_k}\big(u_k u_j\big)\in C\big([-T,T],Q^{l,s}_\Om\big)
\end{equation}
where $\Delta^{-1}\partial_{x_j}\partial_{x_k} : Q^{l,s}_\Om(\R^n)\to Q^{l,s}_\Om(\R^n)$ is
a bounded linear map defined as in \eqref{eq:Delta-inverse}.
In particular, \eqref{eq:P-regularity} and Lemma \ref{lem:Q^{l,s}} imply that $\nabla P\in Q^{l-1,s}_\Om$ has
mean-value zero. Since by the second part of Lemma \ref{lem:PP}, $\PP(u)$ has mean-value zero, we obtain 
from \eqref{eq:pressure3} that
\begin{equation}\label{eq:pressure4}
-\nabla\pp(t)=\PP(u(t)),\quad t\in[-T,T],
\end{equation}
where $\PP(u)\in C\big([-T,T],Q^{l,s}_\Om\big)$ by the first part of Lemma \ref{lem:PP}.
By combining this with \eqref{eq:P-regularity} we conclude that 
$\nabla\pp,\PP(u)\in C\big([-T,T],Q^{l,s}_\Om\big)$, and hence
\begin{equation}\label{eq:P-regularity_final}
\pp\in C\big([-T,T],Q^{l+1,s}_\Om\big)\,.
\end{equation}
Moreover, \eqref{eq:pressure4} implies that \eqref{eq:u-solution} satisfies the equation
\begin{equation}\label{eq:euler2}
\left\{
\begin{array}{l}
u_t+u\cdot\nabla u=\PP(u),\\
u|_{t=0}=u_0,\quad\Div u_0=0.
\end{array}
\right.
\end{equation}
We have the following lemma.

\begin{Lem}\label{lem:equations_equivalence}
Assume that $u\in C\big([-T,T],Q^{l,s}_\Om\big)\cap C^1\big([-T,T],Q^{l-1,s}_\Om\big)$ 
for some $l\ge 2$ and $s>M/2$. Then, $u$ is a solution of the Euler equation \eqref{eq:euler1} 
such that the pressure $\pp(t)$ belongs to $Q^s_\Om(\R^n)$ and has mean-value zero for any $t\in[-T,T]$ if and only if 
$u$ satisfies equation \eqref{eq:euler2}. In both (equivalent) cases, we have that
$\pp\in C\big([-T,T],Q^{l+1,s}_\Om\big)$.
\end{Lem}

Note that in contrast to the solutions of the Euler equation \eqref{eq:euler1}, the solutions of \eqref{eq:euler2} 
are {\em not} assumed to be divergence free on their interval of existence.
By Lemma \ref{lem:equations_equivalence}, a solution of \eqref{eq:euler2} that is divergence free at $t=0$ is divergence free
for any time.

\begin{proof}[Proof of Lemma \ref{lem:equations_equivalence}.]
The direct implication and \eqref{eq:P-regularity_final} are already proven.
Now, assume that $u\in C\big([-T,T],Q^{l,s}_\Om\big)\cap C^1\big([-T,T],Q^{l-1,s}_\Om\big)$ is 
a solution of \eqref{eq:euler2}. 
We will first prove that $\Div u_0=0$ implies that $\Div u(t)=0$ for any $t\in[-T,T]$. 
To this end, we apply $\Div$ to the both sides of \eqref{eq:euler2} to conclude that
\[
\big(\Div u\big)^{\mbox{\bfseries .}}+\Div\big(u\cdot\nabla u\big)=\Div\PP(u)
\]
Then, we use that $\Div\big(u\cdot\nabla u\big)=Q(u)+u\cdot\nabla\big(\Div u\big)$ and
\[
\Div\PP(u)=\Pi_\infty\circ Q(u)+\Pi_\bullet\Div\big(u\cdot\nabla u\big)
\] 
to conclude that $\big(\Div u\big)^{\mbox{\bfseries .}}+\Pi_\infty\Big(u\cdot\nabla\big(\Div u\big)\Big)=0$.
This equation splits into two relations
\begin{align}
\big(\Div u_\infty\big)^{\mbox{\bfseries .}}+\Pi_\infty\Big(u\cdot\nabla\big(\Div u\big)\Big)=0\label{eq:split1}\\
\big(\Div u_\bullet\big)^{\mbox{\bfseries .}}=0\label{eq:split2}
\end{align}
where we set $u_\bullet:=\Pi_\bullet u$ and  $u_\infty:=\Pi_\infty u$. Since $\Div u_0=0$ we have 
\begin{equation}\label{eq:split_divergence_free}
\Div u_\bullet|_{t=0}=0\quad\text{\rm and}\quad\Div u_\infty|_{t=0}=0\,.
\end{equation}
It follows from  \eqref{eq:split2} and the first relation in \eqref{eq:split_divergence_free} that
$\Div u_\bullet(t)=0$ for any $t\in[-T,T]$. Hence, \eqref{eq:split1} becomes
\begin{equation}\label{eq:split1'}
\big(\Div u_\infty\big)^{\mbox{\bfseries .}}+\Pi_\infty\Big(u\cdot\nabla\big(\Div u_\infty\big)\Big)=0.
\end{equation}
We then multiply \eqref{eq:split1'} by $\Div u_\infty$, integrate over the cube $[-T,T]^n\subseteq\R^n$, 
take the limit as $T\to\infty$, and use the Stokes' theorem as in \eqref{eq:energy_computation} to obtain that
\begin{equation}
\Big(\|\Div u_\infty\|_0^2\Big)^{\mbox{\bfseries .}}=
-\lim_{T\to\infty}\frac{1}{(2T)^n}\int_{[-T,T]^n}\big(\Div u_\infty\big)^3\,dx\le C\|\Div u_\infty\|_0^2
\end{equation}
where $C>0$ is chosen so that $C\ge\max\limits_{[-T,T]}|\Div u_\infty|_\infty\ge 0$ (cf. Remark \ref{rem:L^infty})
and (cf. \eqref{eq:Besicovich_product1}),
\[
\|f\|_0\equiv\Big(\sum_{m\in\Z^M}|\widehat{f}_m|^2\Big)^{1/2}=
\left(\lim\limits_{T\to\infty}\frac{1}{(2T)^n}\int_{[-T,T]^n} f^2\,dx\right)^{1/2},\quad f\in Q^s_\Om(\R^n).
\]
By Gronwall's inequality we then obtain that, $\|\Div u_\infty\|_0=0$. Since $\|\cdot\|_0$ is norm, $\Div u(t)=0$ for any $t\in[-T,T]$.

Since $\Div u=0$ it then follows from \eqref{eq:PP}, \eqref{eq:Q}, 
and \eqref{eq:momentum_tensor_form}, that
\begin{eqnarray*}
\PP(u)&=&\nabla\circ\Delta^{-1}\Big(\Pi_\infty\circ Q(u)+\Pi_\bullet\Div\big(u\cdot\nabla u\big)\Big)
=\nabla\circ\Delta^{-1}\Big(\Div\big(u\cdot\nabla u\big)\Big)\\
&=&\nabla\circ\Delta^{-1}\Big(\sum_{k,j=1}^n\partial_{x_j}\partial_{x_k}\big(u_ku_j\big)\Big)
=\nabla\Big(\sum_{k,j=1}^n\Delta^{-1}\partial_{x_j}\partial_{x_k}\big(u_ku_j\big)\Big)=\nabla\pp
\end{eqnarray*}
where $\pp:=\sum_{k,j=1}^n\Delta^{-1}\partial_{x_j}\partial_{x_k}\big(u_ku_j\big)\in C\big([-T,T],Q^{l-1,s}_\Om\big)$
by the boundedness of the map $\Delta^{-1}\partial_{x_j}\partial_{x_k} : Q^{l-1,s}_\Om\to Q^{l-1,s}_\Om$
defined as in \eqref{eq:Delta-inverse}. This implies that $u$ satisfies \eqref{eq:euler2} and the pressure $\pp(t)\in Q^s_\Om(\R^n)$ 
has mean-value zero for any $t\in[-T,T]$.
This completes the proof of the lemma.
\end{proof}

The proof of the following lemma is identical to the proof of Theorem 6.1 in \cite{SunTopalov} (cf. also \cite{EM}) and follows
from Theorem \ref{th:composition} and the embedding $Q^{r,s}_\Om\subseteq C^2_b$ for $r\ge 1$ and $s>\frac{M}{2}+1$
(see Proposition \ref{prop:Q^{l,s}} (iii)) and will be omitted.

\begin{Lem}\label{lem:ode}
Assume that $r\ge 1$ and $s>\frac{M}{2}+1$. Then, for any $T>0$ and $u\in C\big([-T,T],Q^{r,s}_\Om\big)$
there exists a unique solution $\varphi\in C^1\big([-T,T],QD^{r,s}_\Om(\R^n)\big)$ of the differential equation
\begin{equation}\label{eq:ode}
\left\{
\begin{array}{l}
\dt\varphi=u\circ\varphi,\\
\varphi|_{t=0}=\id_{\R^n}.
\end{array}
\right.
\end{equation}
\end{Lem}

We now apply Lemma \ref{lem:ode} to the solution 
\begin{equation}\label{eq:u}
u\in C\big([-T,T],Q^{l,s}_\Om\big)\cap C^1\big([-T,T],Q^{l-1,s}_\Om\big)
\end{equation}
of the Euler equation \eqref{eq:euler1} to obtain a $1$-parameter family of quasi-periodic diffeomorphisms
\begin{equation}\label{eq:phi}
\varphi\in C^1\big([-T,T],QD^{l,s}_\Om(\R^n)\big)
\end{equation}
that satisfies \eqref{eq:ode} for $l\ge 2$ and $s>\frac{M}{2}+1$.
Denote 
\begin{equation}\label{eq:v}
v:=u\circ\varphi\in C\big([-T,T],Q^{l,s}_\Om\big).
\end{equation}
In view of \eqref{eq:ode}, \eqref{eq:u}, \eqref{eq:v}, and the embedding 
$Q^{l,s}_\Om\subseteq C^2_b$ for $l\ge 2$ and $s>\frac{M}{2}+1$ we conclude that 
$u,\varphi,v\in C^1\big([-T,T]\times\R^n,\R^n\big)$ and hence
\begin{equation}\label{eq:v-dot}
\dt{v}=\dt{u}\circ\varphi+[du]\circ\varphi\cdot\dt{\varphi}
\end{equation}
pointwise. This implies that for any $\tau\in[-T,T]$,
\begin{equation}\label{eq:v(s)}
v(\tau)=u_0+\int_0^\tau\Big(\dt{u}\circ\varphi+[du]\circ\varphi\cdot\dt{\varphi}\Big)\,dt
\end{equation}
pointwise. In view of Theorem \ref{th:composition} we then conclude from \eqref{eq:u} and \eqref{eq:phi}
that the integrand in \eqref{eq:v(s)} belongs to $C\big([-T,T],Q^{l-1,s}_\Om\big)$.
This and \eqref{eq:v} imply that the integral in \eqref{eq:v(s)} converges in $Q^{l-1,s}_\Om$, and hence,
\begin{equation}\label{eq:v-regularity}
v\in C\big([-T,T],Q^{l,s}_\Om\big)\cap C^1\big([-T,T],Q^{l-1,s}_\Om\big)\,.
\end{equation}
For $\psi\in QD^{r,s}_\Om(\R^n)$, $r\ge 0$, $s>\frac{M}{2}+1$, consider the right translation of 
of vector fields (or functions) on $QD^{r,s}_\Om(\R^n)$,
\begin{equation}\label{eq:R}
R_\psi : Q^{r,s}_\Om\to Q^{r,s}_\Om,\quad f\mapsto R_\psi(f):=f\circ\psi.
\end{equation}

\begin{Rem}\label{rem:right_translation1}
For simplicity of notation, we will use the same symbol for the right translation 
$R_\psi :QD^{r,s}_\Om(\R^n)\to QD^{r,s}_\Om(\R^n)$ on the group $QD^{r,s}_\Om(\R^n)$, $r\ge 0$, $s>\frac{M}{2}+1$.
Strictly speaking, the map \eqref{eq:R} is the linearization of this map. The both maps are real analytic
(see Remark \ref{rem:right_translation2} below).
\end{Rem}

In view of Theorem \ref{th:composition}, the map \eqref{eq:R} is a linear isomorphism of Banach spaces
with inverse $R_\psi^{-1}=R_{\psi^{-1}}$.
It follows from \eqref{eq:v-dot} and \eqref{eq:ode} that
\begin{align}
\dt{v}&=\dt{u}\circ\varphi+[du]\circ\varphi\cdot\dt{\varphi\nonumber}\\
&=\big(\dt{u}+u\cdot\nabla u\big)\circ\varphi\nonumber\\
&=\big(R_\varphi\circ\PP\circ R_\varphi\big)(v)
\end{align}
where we used that $u$ satisfies the equation \eqref{eq:euler2} 
(cf. Lemma \ref{lem:equations_equivalence}).
Hence, $\varphi\in C^1\big([-T,T],QD^{l,s}_\Om(\R^n)\big)$ and
$v\in C\big([-T,T],Q^{l,s}_\Om\big)\cap C^1\big([-T,T],Q^{l-1,s}_\Om\big)$
satisfy the system of equations
\begin{equation}\label{eq:dynamical_system}
\left\{
\begin{array}{l}
\dt{\varphi}=v,\\
\dt{v}=F(v,\varphi),
\end{array}
\right.
\end{equation}
where, in view of \eqref{eq:PP},
\begin{align}\label{eq:F}
F(v,\varphi)&:=\big(R_\varphi\circ\PP\circ R_{\varphi^{-1}}\big)(v)\nonumber\\
&=\Big(R_\varphi\circ\big(\Delta^{-1}\circ\Pi_\infty\big)\circ R_{\varphi^{-1}}\Big)\circ
\big(R_\varphi\circ\nabla\circ R_{\varphi^{-1}}\big)
\circ\big(R_\varphi\circ Q\circ R_{\varphi^{-1}}\big)\nonumber\\
&+\Big(R_\varphi\circ\big(\Delta^{-1}\circ\nabla\circ\Div\big)\circ\Pi_\bullet\circ R_{\varphi^{-1}}\Big)\circ
\big(R_\varphi\circ D\circ R_{\varphi^{-1}}\big)\,.
\end{align}
It follows from Theorem \ref{th:composition} and Lemma \ref{lem:PP} that
\begin{equation}\label{eq:F-map}
F : Q^{l,s}_\Om\times QD^{l,s}_\Om(\R^n)\to Q^{l,s}_\Om
\end{equation}
is continuous. By the second equality in \eqref{eq:dynamical_system} we then obtain that for $\tau\in[-T,T]$,
$v(\tau)=u_0+\int_0^\tau F(v,\varphi)\,dt$, where the integrand is a continuous function of $t$, and hence,
\[
v\in C^1\big([-T,T],Q^{l,s}_\Om\big).
\]
In this way we proved the direct statement of the following proposition.

\begin{Prop}\label{prop:pde<->ode}
Assume that $l\ge 2$ and $s>\frac{M}{2}+1$. Then, for any $T>0$ there exists 
a one-to-one correspondence between solutions 
\[
u\in C\big([-T,T],Q^{l,s}_\Om\big)\cap C^1\big([-T,T],Q^{l-1,s}_\Om\big)
\]
of the Euler equation \eqref{eq:euler1} such that the pressure $\pp(t)$ belongs to $Q^s_\Om(\R^n)$ and has
mean-value zero for any $t\in[-T,T]$, and solutions 
$(v,\varphi)\in C^1\big([-T,T],Q^{l,s}_\Om\times QD^{l,s}_\Om(\R^n)\big)$ of the system \eqref{eq:dynamical_system}
with initial data $(v,\varphi)|_{t=0}=(u_0,\id_{\R^n})$ where $v=u\circ\varphi$.
\end{Prop}

\noindent The converse statement in Proposition \ref{prop:pde<->ode} follows easily from 
Lemma \ref{lem:equations_equivalence} and Theorem \ref{th:composition}.
Since the arguments are fairly standard we will omit them and concentrate on the new aspects related to the quasi-periodicity
of the solutions.

\medskip

Our main task now is to prove that the dynamical system \eqref{eq:dynamical_system} has a smooth right side.

\begin{Prop}\label{prop:F-smooth}
For $l\ge 2$ and $s>\frac{M}{2}+1$ the map \eqref{eq:F-map} is real analytic.
\end{Prop}

\noindent Theorem \ref{th:euler} will then follow from Proposition \ref{prop:pde<->ode}, Proposition \ref{prop:F-smooth},
and the existence theorems for solutions of ordinary differential equations in Banach spaces.

\medskip

We will prove Proposition \ref{prop:F-smooth} by showing that the conjugated factors appearing in \eqref{eq:F} are real analytic.
The proofs of the following two statements are standard (cf. e.g. \cite[Appendix A]{SunTopalov}).

\begin{Lem}\label{lem:nabla-conjugate}
For $r\ge r_0\ge 1$ and $s>\frac{M}{2}+1$ the map
\[
Q^{r_0,s}_\Om(\R^n)\times QD^{r,s}_\Om(\R^n)\to Q^{r_0-1,s}_\Om(\R^n),
\quad (f,\varphi)\mapsto\big(R_\varphi\circ\nabla\circ R_{\varphi^{-1}}\big)(f),
\]
is real analytic.
\end{Lem}

\begin{proof}[Proof of Lemma \ref{lem:nabla-conjugate}]
In view of the embedding $Q^{r_0,s}_\Om(\R^n)\subseteq C^2_b(\R^n)$, we obtain by a direct computation 
involving \eqref{eq:d_phi_inverse} that for $f\in Q^{r_0,s}_\Om(\R^n)$,
\[
\big(R_\varphi\circ\nabla\circ R_{\varphi^{-1}}\big)(f)=(\nabla f)\cdot[d\varphi]^{-1},
\]
where $(\nabla f)=\big(\p_{x_1}f,...,\p_{x_n}f\big)$. Lemma \ref{lem:nabla-conjugate} then follows from 
the Banach algebra property of $Q^{r_0-1,s}_\Om(\R^n)$, the formula for the inverse of a matrix, and Lemma \ref{lem:division}.
\end{proof}

As a direct consequence of  Lemma \ref{lem:nabla-conjugate} we obtain 

\begin{Lem}\label{lem:Q-conjugate}
Assume that $s>\frac{M}{2}+1$. Then, we have
\begin{itemize}
\item[(i)] For $r\ge r_0\ge 2$ the map
\[
Q^{r_0,s}_\Om(\R^n)\times QD^{r,s}_\Om(\R^n)\to Q^{r_0-2,s}_\Om(\R^n),
\quad (f,\varphi)\mapsto\big(R_\varphi\circ\Delta\circ R_{\varphi^{-1}}\big)(f),
\]
is real analytic.
\item[(ii)] For $r\ge r_0\ge 1$ the maps $Q^{r_0,s}_\Om\times QD^{r,s}_\Om(\R^n)\to Q^{r_0-1,s}_\Om(\R^n)$,
\[
\quad (w,\varphi)\mapsto\big(R_\varphi\circ D\circ R_{\varphi^{-1}}\big)(w)\quad\text{\rm and}\quad
(w,\varphi)\mapsto\big(R_\varphi\circ Q\circ R_{\varphi^{-1}}\big)(w),
\]
\end{itemize}
where $D(w)\equiv w\cdot\nabla w$ and $Q(w)\equiv\tr[dw]^2$, are real analytic.
\end{Lem}

Lemma \ref{lem:Q-conjugate} follows directly from Lemma \ref{lem:nabla-conjugate}, the Banach algebra property of $Q^{r_0-1,s}_\Om$, and
the following easily verified identities
\[
\big(R_\varphi\circ\Delta\circ R_{\varphi^{-1}}\big)(f)=\big(R_\varphi\circ\Div\circ R_{\varphi^{-1}}\big)\circ
\big(R_\varphi\circ\nabla\circ R_{\varphi^{-1}}\big)(f),
\]
\[
\big(R_\varphi\circ\Div\circ R_{\varphi^{-1}}\big)(w)=\tr\big[\big(R_\varphi\circ\nabla\circ R_{\varphi^{-1}}\big)(w)\big],
\]
\[
\big(R_\varphi\circ D\circ R_{\varphi^{-1}}\big)(w)=\big[\big(R_\varphi\circ\nabla\circ R_{\varphi^{-1}}\big)(w)\big]\cdot w,
\]
and
\[
\big(R_\varphi\circ Q\circ R_{\varphi^{-1}}\big)(w)=\tr\big[\big(R_\varphi\circ\nabla\circ R_{\varphi^{-1}}\big)(w)\big]^2,
\]
where $f\in Q^{r_0+1,s}_\Om(\R^n)$, $w=(w_1,...,w_n)^T\in Q^{r_0,s}_\Om$, $R_\varphi\circ\nabla\circ R_{\varphi^{-1}}$
is applied component wise to $w$ to obtain an $n\times n$-matrix, and $(\cdot)^T$ denotes the transpose of a matrix.

The following lemma plays an important role in the proof of Theorem \ref{th:euler}.

\begin{Lem}\label{lem:smoothing}
For $r\ge r_0\ge 0$, $s>\frac{M}{2}+1$, and for any $\tau\ge 0$, the maps
$QD^{r,s}_\Om(\R^n)\times Q^{r_0,s}_\Om(\R^n)\to Q^{r_0+\tau,s}_\Om(\R^n)$
\begin{equation}\label{eq:smoothing1}
QD^{r,s}_\Om(\R^n)\times Q^{r_0,s}_\Om(\R^n)\to Q^{r_0+\tau,s}_\Om(\R^n),\quad
(\varphi,f)\mapsto\big(\Pi_\bullet\circ R_{\varphi^{-1}}\big)(f)
\end{equation}
and
\begin{equation}\label{eq:smoothing2}
QD^{r,s}_\Om(\R^n)\times Q^{r_0,s}_\Om(\R^n)\to Q^{r_0,s}_\Om(\R^n),\quad
(\varphi,f)\mapsto\big(R_{\varphi}\circ\Pi_\bullet\big)(f)
\end{equation}
are real analytic.
\end{Lem}

\begin{Rem}
Since $\tau\ge 0$ in Lemma \ref{lem:smoothing} can be chosen arbitrarily large, 
we conclude that the map \eqref{eq:smoothing1} is a $C^\infty$-smoothing. 
This is {\em not} true for the map \eqref{eq:smoothing2}.
\end{Rem}

\begin{proof}[Proof of Lemma \ref{lem:smoothing}]
Let $r\ge r_0\ge 0$, $s>\frac{M}{2}+1$, and $\tau\ge 0$.
Consider first the map \eqref{eq:smoothing1}. This map is well defined 
and continuous by Lemma \ref{lem:Q_bullet-smooth} and Theorem \ref{th:composition}.
Take $\varphi\in QD^{r,s}_\Om(\R^n)$, $f\in Q^{r_0,s}_\Om(\R^n)$. 
For a given $m\in I_\bullet$ denote by $\widehat{a}_m$ the $m$'th Fourier coefficient of 
$\big(\Pi_\bullet\circ R_{\varphi^{-1}}\big)(f)\in Q^{r_0+\tau,s}_\Om(\R^n)$.
By formula \eqref{eq:f_m} in Lemma \ref{lem:qp_distributions1} we have
\begin{align}\label{eq:a_m}
\widehat{a}_m&=\lim_{T\to\infty}\frac{1}{(2T)^n}\int_{[-T,T]^n} f(\varphi^{-1}(y))\,e^{-i(\Lambda_m,y)}\,dy\nonumber\\
&=\lim_{T\to\infty}\frac{1}{(2T)^n}\int_{[-T,T]^n}f(x)\,e^{-i(\Lambda_m,\varphi(x))}\det[d_x\varphi]\,dx
\end{align}
where we change the variables in the integral and then use that the integrand in \eqref{eq:a_m} is 
uniformly bounded on $\R^n$ and
\[
\int_{D(T)}\big|f(x)\det[d_x\varphi]\big|\,dx=O\big(T^{n-1}\big)\quad\text{\rm as}\quad T\to\infty
\]
where $D(T)$ denotes the symmetric difference $\varphi^{-1}\big([-T,T]^n\big)\triangle\big([-T,T]^n\big)$.
The later follows from the fact that for $T>B\ge 0$ we have the inclusion of sets in $\R^n$,
\[
\big[-(T-B),T-B\big]^n\subseteq\varphi^{-1}\big([-T,T]^n\big)\subseteq
\big[-(T+B),T+B\big]^n,
\]
where $B\ge|h|_\infty$ is a constant independent of $T>B$ and $\varphi^{-1}=\id_{\R^n}+h$ with $h\in Q^{r,s}_\Om$.
By Taylor's formula for any $x\in\R^n$,
\begin{align}\label{eq:taylor_expansion}
e^{-i(\Lambda_m,\varphi(x))}&=e^{-i(\Lambda_m,x)}\,e^{-i(\Lambda_m,g(x))}\nonumber\\
&=e^{-i(\Lambda_m,x)}\,\sum_{\alpha}\frac{(-i)^{|\alpha|}}{\alpha!}\,\Lambda_m^\alpha g(x)^\alpha
\end{align}
where $\alpha\in\Z_{\ge 0}^n$ is a multi-index and we used the Taylor's expansion of the exponent
$e^{\sum_{j=1}^n x_j}=\sum_{\alpha}\frac{x^\alpha}{\alpha!}$ at $x=0$.
Note that the series in \eqref{eq:taylor_expansion} converges absolutely in $L^\infty(\R^n)$.
By combining this with \eqref{eq:a_m} we then obtain that
\begin{equation}\label{eq:a}
\big(\widehat{a}_m\big)_{m\in I_\bullet}=\sum_{\alpha}\frac{(-i)^{|\alpha|}}{\alpha!}\,
\Big(\Lambda_m^\alpha\widehat{\big(M\,g^\alpha\big)}_m\Big)_{m\in I_\bullet}
\end{equation}
where we set $M:=f\,\det[d\varphi]$.
Recall that $|\Lambda_m|\le 1$ for $m\in I_\bullet$. This and \eqref{eq:a} imply that
\begin{align}
\big\|\big(\Pi_\bullet\circ R_{\varphi^{-1}}\big)(f)\big\|_{r_0+\tau,s}&\le 2^{\frac{r_0+\tau}{2}}\,
\sum_{\alpha}\frac{\|M\,g^\alpha\|_s}{\alpha!}\nonumber\\
&\le 2^{\frac{r_0+\tau}{2}}\,\|M\|_s\sum_{\alpha}\frac{(C_s\|g_1\|_s)^{\alpha_1}...(C_s\|g_n\|_s)^{\alpha_n}}{\alpha!}\nonumber\\
&\le 2^{\frac{r_0+\tau}{2}}\,\|M\|_s\,e^{C_s\|g\|_s},\quad \|g\|_s\equiv\max_{1\le j\le n}\|g_j\|_s,\label{eq:a-estimate}
\end{align}
where $C\equiv C_s>0$ is the constant appearing in the Banach algebra inequality for $Q^s_\Om(\R^n)$ (see Remark \ref{rem:banach_algebra}).
Note that the estimate \eqref{eq:a-estimate} also holds for $g\in Q^{r,s}_{\Om,\C}$ and $f\in Q^{r_0,s}_{\Om,\C}(\R^n)$
(see Remark \ref{rem:complex_spaces}).
This implies that the series in \eqref{eq:a} converges absolutely on bounded sets in $Q^{r_0+\tau,s}_{\Om,\C}(\R^n)$.
Since the terms appearing in the series \eqref{eq:a} are bounded (non-homogeneous) polynomial functions of $g\in Q^{r,s}_{\Om,\C}$ and 
$f\in Q^{r_0,s}_{\Om,\C}(\R^n)$, we conclude from Weierstrass theorem that the map \eqref{eq:smoothing1} is real analytic.

The \eqref{eq:smoothing2} is treated in a similar fashion. 
In fact, it follows immediately from Lemma \ref{lem:Q_bullet-smooth} and Theorem \ref{th:composition}
that the map is well-defined and $C^\infty$-smooth. Let us show that the map is real analytic.
Take $\varphi\in QD^{r,s}_\Om(\R^n)$ and $f\in Q^{r_0,s}_\Om(\R^n)$ where 
$r\ge r_0\ge 0$ and $s>\frac{M}{2}+1$. Then,
\begin{equation}\label{eq:second_map}
\big(R_{\varphi}\circ\Pi_\bullet\big)(f)=\sum_{m\in I_\bullet}\widehat{f}_m e^{i(\Lambda_n,x)}\,e^{i(\Lambda_m,g(x))}
\end{equation}
where $g\in Q^{r,s}_\Om$ and $\varphi=\id_{\R^n}+g$. In view of \eqref{eq:taylor_expansion} we obtain that 
for any $m\in I_\bullet$,
\begin{equation}\label{eq:exp}
e^{i(\Lambda_m,g(x))}=\sum_{\alpha}\frac{(-i)^{|\alpha|}}{\alpha!}\,\Lambda_m^\alpha g(x)^\alpha.
\end{equation}
One easily sees from the Banach algebra property of $Q^{r,s}_{\Om,\C}$ that for any $g\in Q^{r,s}_{\Om,\C}$ 
and for any $m\in I_\bullet$ the series \eqref{eq:exp} converges absolutely in $Q^{r,s}_{\Om,\C}(\R^n)$ and
\begin{equation}\label{eq:exp-estimate}
\big\|e^{i(\Lambda_m,g)}\big\|_{r,s}\le e^{C_s\|g\|_s}.
\end{equation}
(Note that this is the place where we use that $|\Lambda_m|\le 1$ for $m\in I_\bullet$.)
This implies that for any $m\in I_\bullet$ the map
\[
Q^{r,s}_{\Om,\C}\to Q^{r,s}_{\Om,\C}(\R^n),\quad g\mapsto e^{i(\Lambda_m,g)},
\]
is analytic. By \eqref{eq:exp-estimate} and the fact that 
$(\widehat{f}_m)_{m\in I_\bullet}\in\ell^1(I_\bullet,\C)$ for $f\in Q^{r_0,s}_{\Om,\C}(\R^n)$,
the series in \eqref{eq:second_map} converges absolutely and uniformly on bounded sets in 
the space $Q^{r_0,s}_{\Om,\C}(\R^n)$. Hence, by the Weierstrass theorem, 
the map \eqref{eq:smoothing2} is real analytic. 
This completes the proof of Lemma \ref{lem:smoothing}.
\end{proof}

As an immediate corollary we obtain

\begin{Coro}\label{lem:Pi-conjugate}
Assume that $r\ge 0$ and $s>\frac{M}{2}+1$. Then, the maps 
$QD^{r,s}_\Om(\R^n)\times Q^{r,s}_\Om(\R^n)\to Q^{r,s}_\Om(\R^n)$,
\begin{equation}\label{eq:Pi-conjugate}
(\varphi,f)\mapsto\big(R_\varphi\circ\Pi_\bullet\circ R_{\varphi^{-1}}\big)(f)
\quad\text{\rm and}\quad
(\varphi,f)\mapsto\big(R_\varphi\circ\Pi_\infty\circ R_{\varphi^{-1}}\big)(f)
\end{equation}
are real analytic.
\end{Coro}

\begin{proof}[Proof of Corollary \ref{lem:Pi-conjugate}]
The analyticity of the first map in \eqref{eq:Pi-conjugate} follows directly from
Lemma \ref{lem:smoothing}. Then, the analyticity of the second map in \eqref{eq:Pi-conjugate}
follows from the fact that 
\[
\big(R_\varphi\circ\Pi_\bullet\circ R_{\varphi^{-1}}\big)(f)+\big(R_\varphi\circ\Pi_\infty\circ R_{\varphi^{-1}}\big)(f)=f\
\]
for any $f\in Q^{r,s}_\Om(\R^n)$.
\end{proof}

Our next statement concerns the map
\begin{equation}\label{eq:factor1}
Q^{r,s}_\Om\times QD^{r,s}_\Om(\R^n)\to Q^{r,s}_\Om,\quad 
(w,\varphi)\mapsto\Big(R_\varphi\circ\big(\Delta^{-1}\circ\nabla\circ\,\Div\big)\circ\Pi_\bullet\circ R_{\varphi^{-1}}\Big)(w)
\end{equation}
that appears as a factor in \eqref{eq:F}. We have

\begin{Lem}\label{lem:factor1}
For $r\ge 0$ and $s>\frac{M}{2}+1$ the map \eqref{eq:factor1} is real analytic.
\end{Lem}

\begin{proof}[Proof of Lemma \ref{lem:factor1}]
Take $r\ge 0$ and $s>\frac{M}{2}+1$.
For any $w\in Q^{r,s}_\Om$ and $\varphi\in QD^{r,s}_\Om(\R^n)$ we have that
\[
R_\varphi\circ\big(\Delta^{-1}\circ\nabla\circ\,\Div\big)\circ\,\Pi_\bullet\circ R_{\varphi^{-1}}(w)=
\big(R_\varphi\circ\,\Pi_\bullet\big)\circ\big(\Delta^{-1}\circ\nabla\circ\,\Div\big)\circ\big(\Pi_\bullet\circ R_{\varphi^{-1}}\big)(w).
\]
Now, the proof of the lemma follows directly from Lemma \ref{lem:smoothing} and 
the boundedness of the map \eqref{eq:Delta-inverse-map}.
\end{proof}

Finally, for $r\ge 2$ and $s>\frac{M}{2}+1$ consider the map (cf. \eqref{eq:F}),
\begin{equation}\label{eq:factor2}
Q^{r-2,s}_\Om\times QD^{r,s}_\Om(\R^n)\to Q^{r,s}_\Om,\quad 
(w,\varphi)\mapsto\Big(R_\varphi\circ\big(\Delta^{-1}\circ\Pi_\infty\big)\circ R_{\varphi^{-1}}\Big)(w).
\end{equation}
Note that by Lemma \ref{lem:Delta-isomorphism} and Theorem \ref{th:composition} the map 
\eqref{eq:factor2} is well defined and continuous. We will prove

\begin{Lem}\label{lem:factor2}
For $r\ge 2$ and $s>\frac{M}{2}+1$ the map \eqref{eq:factor2} is real analytic.
\end{Lem}

For the proof of this lemma we need a preparation.
Assume that $r\ge 0$ and $s>\frac{M}{2}+1$.
For a given $\varphi\in QD^{r,s}_\Om(\R^n)$ consider the commutative diagram
\begin{equation}\label{eq:C-diagram}
\begin{tikzcd}
Q^{r,s}_{\Om,\infty}\arrow[swap]{dr}{C_\varphi}\arrow{r}{\Pi_{\infty,\varphi}}&
R_\varphi(Q^{r,s}_{\Om,\infty})\arrow{d}{\Pi_\infty}\\
&Q^{r,s}_{\Om,\infty}
\end{tikzcd}
\end{equation}
where $R_\varphi(Q^{r,s}_{\Om,\infty})$ denotes the image of the subspace 
$Q^{r,s}_{\Om,\infty}\subseteq Q^{r,s}_\Om$ with respect to the right translation 
$R_\varphi : Q^{r,s}_\Om\to Q^{r,s}_\Om$,
\[
\Pi_{\infty,\varphi}:=R_\varphi\circ\Pi_\infty\circ R_{\varphi^{-1}} : 
Q^{r,s}_\Om\to R_\varphi(Q^{r,s}_{\Om,\infty}),
\] 
and the map $C_\varphi$ is defined by the diagram.
Since the right translation $R : Q^{r,s}_\Om\to Q^{r,s}_\Om$ is an isomorphism
of Banach spaces we see that $R_\varphi(Q^{r,s}_{\Om,\infty})$ is a closed subspace in $Q^{r,s}_\Om$.
We have

\begin{Lem}\label{lem:C-inverse}
For $r\ge 0$ and $s>\frac{M}{2}+1$ there exists an open neighborhood $U$ of the identity in 
$QD^{r,s}_\Om(\R^n)$ such that for any $\varphi\in U$ the map 
$C_\varphi : Q^{r,s}_{\Om,\infty}\to Q^{r,s}_{\Om,\infty}$
is a linear isomorphism and 
\begin{equation}\label{eq:C-inverse}
Q^{r,s}_{\Om,\infty}\times U\to Q^{r,s}_{\Om,\infty},\quad
(w,\varphi)\mapsto C_\varphi^{-1}(w),
\end{equation}
is real analytic. 
\end{Lem}

\begin{proof}[Proof of Lemma \ref{lem:C-inverse}]
The lemma follows from the inverse function theorem.
Indeed, by Corollary \ref{lem:Pi-conjugate} the map
\[
C : Q^{r,s}_{\Om,\infty}\times QD^{r,s}_\Om(\R^n)\to Q^{r,s}_{\Om,\infty}\times QD^{r,s}_\Om(\R^n),\quad
(w,\varphi)\stackrel{T}{\mapsto}\big(\Pi_\infty\circ\Pi_{\infty,\varphi}(w),\varphi\big),
\]
is real analytic. For the the differential
\[
d_{(0,\id_{\R^n})}C : Q^{r,s}_{\Om,\infty}\times Q^{r,s}_\Om\to
Q^{r,s}_{\Om,\infty}\times Q^{r,s}_\Om
\]
of $C$ at the point $(0,\id_{\R^n})$ we have 
$d_{(0,\id_{\R^n})}C=\big(\id_{Q^{r,s}_{\Om,\infty}},\id_{Q^{r,s}_\Om}\big)$.
Hence, by the inverse function theorem there exist an open neighborhood $U$
of the identity in $QD^{r,s}_\Om(\R^n)$ and open neighborhoods $V$ and $W$ of zero
in $Q^{r,s}_{\Om,\infty}$ such that the map
\[
V\times U\to W\times U,\quad (w,\varphi)\mapsto(C_\varphi(w),\varphi),
\]
is real analytic diffeomorphism. Since for any given $\varphi\in U$ the map
\begin{equation}\label{eq:C}
C_\varphi : Q^{r,s}_{\Om,\infty}\to Q^{r,s}_{\Om,\infty}
\end{equation}
is a bounded linear map we then conclude from the open mapping theorem that
\eqref{eq:C} is a linear isomorphism and that \eqref{eq:C-inverse} is real analytic.
\end{proof}

By combining Lemma \ref{lem:C-inverse} with the commutative diagram \eqref{eq:C-diagram}
we obtain

\begin{Coro}\label{coro:C-chart}
Assume that $r\ge 0$, $s>\frac{M}{2}+1$. Then, there exists  and open neighborhood $U$ of the 
identity in $QD^{r,s}_\Om(\R^n)$ such that the statement of Lemma \ref{lem:C-inverse} holds and 
for any $\varphi\in U$ the restriction of the map
\begin{equation}\label{eq:C-chart}
Q^{r,s}_\Om\to Q^{r,s}_{\Om,\infty},\quad w\mapsto C_\varphi^{-1}\circ\Pi_\infty(w),
\end{equation}
to the subspace $R_\varphi(Q^{r,s}_{\Om,\infty})$ is the inverse of the map
\begin{equation}\label{eq:Pi-chart}
\Pi_{\infty,\varphi}\equiv R_\varphi\circ\Pi_\infty\circ R_{\varphi^{-1}} : 
Q^{r,s}_{\Om,\infty}\to R_\varphi(Q^{r,s}_{\Om,\infty}).
\end{equation}
The map
\[
Q^{r,s}_\Om\times U\to Q^{r,s}_{\Om,\infty},\quad
(w,\varphi)\mapsto C_\varphi^{-1}\circ\Pi_\infty(w),
\]
is real analytic.
\end{Coro}

\begin{Rem}
This corollary allows us to ``coordinatize'' the subspace $R_\varphi(Q^{r,s}_{\Om,\infty})$
of $Q^{r,s}_\Om$ by the subspace $Q^{r,s}_{\Om,\infty}$ via the maps
\eqref{eq:C-chart} and \eqref{eq:Pi-chart} that depend analytically on $\varphi\in U$.
Note that we cannot use the right translation $R_\varphi : Q^{r,s}_{\Om,\infty}\to R_\varphi(Q^{r,s}_{\Om,\infty})$
and its inverse for this purpose since they do {\em not} depend smoothly on $\varphi\in U$.
\end{Rem}

Now assume that $r\ge 2$ and let $U$ be the open neighborhood of the identity in $QD^{r,s}_\Om(\R^n)$ 
from Corollary \ref{coro:C-chart}. For $\varphi\in U$ consider the commutative diagram
\begin{equation}\label{eq:Delta-diagram}
\begin{tikzcd}
R_\varphi(Q^{r,s}_{\Om,\infty})\arrow{r}{\Delta_\varphi}&R_\varphi(Q^{r-2,s}_{\Om,\infty})\\
Q^{r,s}_{\Om,\infty}\arrow{u}{\Pi_{\infty,\varphi}}\arrow{r}{\widetilde{\Delta}_\varphi}
&Q^{r-2,s}_{\Om,\infty}\arrow[swap]{u}{\Pi_{\infty,\varphi}}
\end{tikzcd}
\end{equation}
where 
\begin{equation*}
\Delta_\varphi:=R_\varphi\circ\Delta\circ R_{\varphi^{-1}}
\end{equation*}
and the map $\widetilde{\Delta}_\varphi$ is defined by the diagram.
Note that all arrows in \eqref{eq:Delta-diagram} are linear isomorphisms by 
Lemma \ref{lem:Delta-isomorphism} and Corollary \ref{coro:C-chart}.
It then follows from Lemma \ref{lem:Q-conjugate} and Corollary \ref{coro:C-chart} that
the map $\widetilde{\Delta}_\varphi : Q^{r,s}_{\Om,\infty}\to Q^{r-2,s}_{\Om,\infty}$
is well defined and the map
\[
Q^{r,s}_{\Om,\infty}\times U\to Q^{r-2,s}_{\Om,\infty}\times U,\quad
(w,\varphi)\mapsto\big(\widetilde{\Delta}_\varphi(w),\varphi\big),
\]
is real analytic. By shrinking the neighborhood $U$ of the identity if necessary and 
arguing as in the proof of Lemma \ref{lem:C-inverse} one concludes from the inverse function theorem 
that the map 
\begin{equation}\label{eq:Delta-tilde-inverse}
Q^{r-2,s}_{\Om,\infty}\times U\to Q^{r,s}_{\Om,\infty}\quad
(w,\varphi)\mapsto\widetilde{\Delta}^{-1}_\varphi(w),
\end{equation}
is real analytic. We summarize this in the following

\begin{Lem}\label{lem:Delta-tilde-inverse}
Assume that $r\ge 2$ and $s>\frac{M}{2}+1$. Then, there exists an open neighborhood $U$ of the identity in 
$QD^{r,s}_\Om(\R^n)$ such that for any $\varphi\in U$ the map 
$\widetilde{\Delta}_\varphi : Q^{r,s}_{\Om,\infty}\to Q^{r-2,s}_{\Om,\infty}$
is a linear isomorphism and \eqref{eq:Delta-tilde-inverse} is real analytic.
\end{Lem}

We are now ready to proof Lemma \ref{lem:factor2}.

\begin{proof}[Proof of Lemma \ref{lem:factor2}]
Assume that $r\ge 2$ and $s>\frac{M}{2}+1$.
As noted above, the map \eqref{eq:factor2} is well defined and continuous by Lemma \ref{lem:Delta-isomorphism}
and Theorem \ref{th:composition}. We will first prove that \eqref{eq:factor2} is analytic for $\varphi$ in 
an open neighborhood of the identity in $QD^{r,s}_\Om(\R^n)$.
To this end, consider the open neighborhood $U$ of the identity in $QD^{r,s}_\Om(\R^n)$ from Corollary \ref{coro:C-chart} and
Lemma \ref{lem:Delta-tilde-inverse}. For $\varphi\in U$ consider the map
\begin{equation}\label{eq:I}
Q^{r-2,s}_\Om\to Q^{r,s}_\Om,\quad w\mapsto
\mathcal{I}_\varphi(w):=\Pi_{\infty,\varphi}\circ\widetilde{\Delta}_\varphi^{-1}\circ\big(C_\varphi^{-1}\circ\Pi_\infty\big)(w)
\end{equation}
where $C_\varphi^{-1}\circ\Pi_\infty : Q^{r-2,s}_\Om\to Q^{r-2,s}_{\Om,\infty}$ and 
$\widetilde{\Delta}_\varphi : Q^{r,s}_{\Om,\infty}\to Q^{r-2,s}_{\Om,\infty}$ are the maps
\eqref{eq:C-chart} and \eqref{eq:Delta-tilde-inverse} for a fixed second argument $\varphi\in U$.
It follows from the commutative diagram \eqref{eq:Delta-diagram}
and Corollary \ref{coro:C-chart} that the map \eqref{eq:I} when restricted to the subspace $R_\varphi(Q^{r-2,s}_{\Om,\infty})$
of $Q^{r-2,s}_\Om$ is the inverse of the map 
$\Delta_\varphi : R_\varphi(Q^{r,s}_{\Om,\infty})\to R_\varphi(Q^{r-2,s}_{\Om,\infty})$.\footnote{Recall that all arrows in 
diagram \eqref{eq:Delta-diagram} are linear isomorphisms.}
Hence,
\[
\mathcal{I}\big|_{R_\varphi(Q^{r-2,s}_{\Om,\infty})}\equiv\Delta_\varphi^{-1}.
\]
Note that the map
\[
Q^{r-2,s}_\Om\times U\to Q^{r,s}_\Om,\quad (w,\varphi)\mapsto\mathcal{I}_\varphi(w),
\]
is real analytic by Corollary \ref{coro:C-chart} and Lemma \ref{lem:Delta-tilde-inverse}, and the map
\[
Q^{r-2,s}_\Om\times U\to Q^{r-2,s}_\Om,
\quad(w,\varphi)\mapsto\Pi_{\infty,\varphi}(w)\equiv R_\varphi\circ\Pi_\infty\circ R_{\varphi^{-1}}(w),
\]
is real analytic by Lemma \ref{lem:Pi-conjugate}. By combining this with the fact that
\[
R_\varphi\circ\big(\Delta^{-1}\circ\Pi_\infty\big)\circ R_{\varphi^{-1}}(w)=
\mathcal{I}_\varphi\circ\Pi_{\infty,\varphi}(w)
\]
for any $w\in Q^{r-2,s}_\Om$ and $\varphi\in U$ we conclude that the map \eqref{eq:factor2}
is real analytic. This completes the proof of Lemma \ref{lem:factor2} for $\varphi$ in an open neighborhood
$U$ of the identity in $QD^{r,s}_\Om(\R^n)$.

Let us now consider the general case. Take an arbitrary $\varphi_0\in QD^{r,s}_\Om(\R^n)$ and
consider the open neighborhood $V:=R_{\varphi_0}(U)$ of $\varphi_0$ in
$\varphi_0\in QD^{r,s}_\Om(\R^n)$ where $U$ is the open neighborhood of the identity in $QD^{r,s}_\Om(\R^n)$
considered above and let $S : Q^{r-2,s}_\Om\times QD^{r,s}_\Om(\R^n)\to Q^{r,s}_\Om$ is the map \eqref{eq:factor2}.

\begin{Rem}\label{rem:right_translation2}
Note that for a given $\varphi_0\equiv\id_{\R^n}+f_0\in QD^{r,s}_\Om(\R^n)$, $r\ge 0$, $s>\frac{M}{2}+1$, one has
\[
R_{\varphi_0} : QD^{r,s}_\Om(\R^n)\to QD^{r,s}_\Om(\R^n),\quad
\varphi\mapsto R_{\varphi_0}\big(\id_{\R^n}+f\big)=\id_{\R^n}+\big(f_0+f\circ\varphi_0\big), 
\]
where $\varphi\equiv\id_{\R^n}+f\in QD^{r,s}_\Om(\R^n)$. This implies that, in coordinates, the right translation on a fixed
element $\varphi_0$ in $QD^{r,s}_\Om(\R^n)$ is identified with the bounded (Theorem \ref{th:composition}) affine linear map 
$f\mapsto f_0+f\circ\varphi_0$, $Q^{r,s}_\Om\to Q^{r,s}_\Om$. Since $(R_{\varphi_0})^{-1}=R_{\varphi_0^{-1}}$ we
then conclude that $R_{\varphi_0} : QD^{r,s}_\Om(\R^n)\to QD^{r,s}_\Om(\R^n)$ is a bi-analytic diffeomorphism.
The rights translation \eqref{eq:R} of vector fields $R_{\varphi_0} : Q^{r,s}_\Om\to Q^{r,s}_\Om$ is a bounded linear map, 
and hence analytic (see Remark \ref{rem:right_translation1}).
\end{Rem}

\noindent For any $\psi\in V$ and $w\in Q^{r-2,s}_\Om$ we have
\[
R_\psi\circ\big(\Delta^{-1}\circ\Pi_\infty\big)\circ R_{\psi^{-1}}(w)=
R_{\varphi_0}\circ \Big(R_\varphi\circ\big(\Delta^{-1}\circ\Pi_\infty\big)\circ R_{\varphi^{-1}}\Big)\circ R_{\varphi_0^{-1}}(w).
\]
where $\varphi=\psi\circ\varphi_0^{-1}=R_{\varphi_0^{-1}}(\psi)$. This implies that for any $\psi\in V$ and $w\in Q^{r-2,s}_\Om$,
\begin{equation}\label{eq:S-symmetry}
S(w,\psi)=R_{\varphi_0}\Big(S\big|_U\big(R_{\varphi_0^{-1}}(w),R_{\varphi_0^{-1}}(\psi)\big)\Big),
\end{equation}
where $S|_U : Q^{r-2,s}_\Om\times U\to Q^{r,s}_\Om$ is the restriction of $S$ to $Q^{r-2,s}_\Om\times U$.
Since the map $S|_U$ is real analytic we then conclude from Remark \ref{rem:right_translation2} that the restriction $S|_V$ of $S$
to $Q^{r-2,s}_\Om\times V$ is real analytic. This completes the proof of the lemma.
\end{proof}

\begin{proof}[Proof of Proposition \ref{prop:F-smooth}]
Proposition \ref{prop:F-smooth} now follows from \eqref{eq:F}, Lemma \ref{lem:nabla-conjugate}, 
Lemma \ref{lem:Q-conjugate}, Lemma \ref{lem:factor1}, and Lemma \ref{lem:factor2}.
\end{proof}

We are now ready to prove Theorem \ref{th:euler}.

\begin{proof}[Proof of Theorem \ref{th:euler}]
Theorem \ref{th:euler} follow from Proposition \ref{prop:pde<->ode}, Proposition \ref{prop:F-smooth},
and the theorems on the existence and the dependence on the initial data of solutions of ordinary differential equations in 
Banach spaces (cf. \cite{Lang}).
\end{proof}

In fact, since by Proposition \ref{prop:F-smooth} the right hand side of \eqref{eq:ode} is real analytic on
$Q^{l,s}_\Om\times QD^{l,s}_\Om$, we obtain from the theorem on the analytic dependence of solutions
of analytic vector fields in Banach spaces (see e.d. \cite{Dieudonne}) the following proposition.

\begin{Prop}\label{prop:analyticity}
Assume that $l\ge 2$, $s>\frac{M}{2}+1$. Then for any $\rho>0$ there exists $T>0$ such that for any 
divergence free vector field $u_0\in B_{Q^{l,s}_\Om}(\rho)$ there exists a unique solution
$(v,\varphi)\in C^1\big([-T,T],Q^{l,s}_\Om\times QD^{l,s}_\Om(\R^n)\big)$ of \eqref{eq:dynamical_system} with initial 
data $(v,\varphi)|_{t=0}=(u_0,\id_{\R^n})$. 
The solution depends analytically on the initial data in the sense that the map
\[
(-T,T)\times B_{Q^{l,s}_\Om}(\rho)\to Q^{l,s}_\Om\times QD^{l,s}_\Om(\R^n),\quad 
(t,u_0)\mapsto\big(v(t;u_0),\varphi(t;u_0)\big),
\]
is real analytic.
\end{Prop}

Since by Proposition \ref{prop:pde<->ode} the solution \eqref{eq:u-solution} of the Euler equation \eqref{eq:euler} and 
the solution $(v,\varphi)\in C^1\big([-T,T],Q^{l,s}_\Om\times QD^{l,s}_\Om(\R^n)\big)$ in Proposition \ref{prop:analyticity} above 
are related by $v=u\circ\varphi$, we conclude that $\dt\varphi=u\circ\varphi$, $\varphi|_{t=0}=\id_{\R^n}$. 
As a consequence we obtain Corollary \ref{coro:analyticity} stated in the Introduction.

\begin{proof}[Proof of Corollary \ref{coro:u_m-analytic}]
Let $u\in C\big((-T_1,T_2),Q^{l,s}_\Om\big)\cap C^1\big((-T_1,T_2),Q^{l-1,s}_\Om\big)$ be the solution of 
the Euler equation \eqref{eq:euler} on its maximal time of existence. Then, by arguing as in \eqref{eq:a_m} and using that
$u=v\circ\varphi^{-1}$ we obtain that for any $t\in(-T_1,T_2)$,
\begin{align}
\widehat{u}_m(t)&:=\lim_{T\to\infty}\frac{1}{(2T)^n}\int_{[-T,T]^n}v\big(t,\varphi^{-1}(t,y)\big) e^{-i(\Lambda_m,y)}\,dy\nonumber\\
&=\lim_{T\to\infty}\frac{1}{(2T)^n}\int_{[-T,T]^n}v\big(t,x\big) e^{-i\big(\Lambda_m,\varphi(t,x)\big)}\,dx\nonumber\\
&=\Big\langle v(t) e^{-i(\Lambda_m,f(t))},e^{-i(\Lambda_m,\cdot)}\Big\rangle_0\label{eq:u_m}
\end{align}
where $\varphi(t)=\id_{\R^n}+f(t)$, $f(t)\in Q^{l,s}_\Om$ and for $h,g\in Q^s_{\Om,\C}(\R^n)$,
\begin{equation}\label{eq:Besicovich_product3}
\langle g,h\rangle_0:=\lim_{T\to\infty}\frac{1}{(2T)^n}\int_{[-T,T]^n}g(x)h(x)\,dx=
\sum_{m\in\Z^M}\widehat{g}_m\widehat{h}_{-m}.
\end{equation}
Since the complex bi-linear form \eqref{eq:Besicovich_product3} is bounded in $Q^s_{\Om,\C}$, the corollary follows from 
\eqref{eq:u_m}, Proposition \ref{prop:analyticity}, and the Banach algebra property in $Q^{l,s}_{\Om,\C}(\R^n)$.
\end{proof}

\appendix

\section{Appendix}
In this Appendix we prove several technical lemmas. The first one is used in Section \ref{sec:spaces}.

\begin{Lem}\label{lem:uniqueness}
Let $(\lambda_j)_{j\in J}$ be a sequence of points in $\R^n$ such that $\lambda_j\ne\lambda_k$ for $j\ne k$ and
$J$ is a countable set of indices. Assume that the series $\sum_{j\in J}\widehat{a}_j e^{i (\lambda_j,x)}$, $\widehat{a}_j\in\C$, 
converges to zero in $S'(\R^n)$ independently of the order of summation. Then $\widehat{a}_j=0$ for any $j\in J$.
\end{Lem}

\begin{proof}[Proof of Lemma \ref{lem:uniqueness}]
By applying the Fourier transform $\mathcal{F} : S'(\R^n)\to S'(\R^n)$ to the series  $\sum_{j\in J}\widehat{a}_j e^{i (\lambda_j,x)}$ 
we see that 
\begin{equation}\label{eq:sum_of_deltas}
\sum_{j\in J}\widehat{a}_j \delta(\xi-\lambda_j) = 0
\end{equation}
where the sum converges to zero in $S'(\R^n)$ independently of the order of summation and $\delta(\xi-\lambda_j)$ is 
the Dirac delta function at the point $\lambda_j$. For a given $\rho>0$ take $\psi\in C^\infty_c(\R^n)$ such that $\psi(\xi)=1$ for 
$|\xi|\le\rho$, $\psi(\xi)=0$ for $|\xi|\ge 2\rho$, and $0\le\psi(\xi)\le 1$ for $\rho\le|\xi|\le 2\rho$.
Then, in view of \eqref{eq:sum_of_deltas}, $\sum_{j\in J}\widehat{a}_j\psi(\lambda_j)$ converges to zero independently
of the order of summation. This implies that it converges absolutely. Since, $\widehat{a}_j\psi(\lambda_j)=\widehat{a}_j$
for $j\in J_{\le\rho}:=\big\{j\in J\,\big|\,|\lambda_j|\le\rho\big\}$, we conclude that for any $\rho>0$ 
there exists a real constant $C_\rho>0$ such that
\begin{equation}\label{eq:l^1_loc}
\sum_{j\in J_{\le\rho}}|\widehat{a}_j|\le C_\rho<\infty.
\end{equation}
Denote $f(\xi):=\sum_{j\in J}\widehat{a}_j \delta(\xi-\lambda_j)\in S'(\R^n)$ and take $R>0$.
Let $\chi_\varepsilon\in C^\infty_c(\R^n)$ be such that
$\chi_\varepsilon(\xi)=0$ for $|\xi|\ge 2 R$, $0\le\chi_\varepsilon(\xi)\le 1$ for $\xi\in\R^n$, and 
\begin{equation}\label{eq:characteristic_function}
\chi_\varepsilon\to\chi_{\le R}\quad\text{as}\quad\varepsilon\to 0+
\end{equation}
pointwisely, where $\chi_{\le R}$ is the characteristic function of the closed disk $\big\{\xi\in\R^n\,\big|\,|\xi|\le R\big\}$.
In view of \eqref{eq:l^1_loc}, \eqref{eq:characteristic_function}, and Lebesgue's dominated convergence theorem, 
for any $\varphi\in S(\R^n)$,
\begin{equation}\label{eq:epsilon_limit}
\langle f,\chi_\varepsilon\varphi\rangle=\sum_{j\in J_{\le 2 R}}\widehat{a}_j \chi_\varepsilon(\lambda_j)\varphi(\lambda_j)\to
\langle f_{\le R},\varphi\rangle
\end{equation}
where $f_{\le R}(\xi):=\sum_{j\in J_{\le R}}\widehat{a}_j\delta(\xi-\lambda_j)$.
Note that in view of \eqref{eq:l^1_loc}, $f_{\le R}$ is a well-defined distribution in $S'(\R^n)$.
Since, by \eqref{eq:sum_of_deltas}, $\langle f,\chi_\varepsilon\varphi\rangle=0$, we conclude from \eqref{eq:epsilon_limit}
that $f_{\le R}=0$ in $S'(\R^n)$. Applying the inverse Fourier transforn to this equality, we obtain that
\[
\sum_{j\in J_{\le R}}\widehat{a}_j e^{i (\lambda_j,x)}=0
\]
in $S'(\R^n)$. Since, by \eqref{eq:l^1_loc}, the series above converges uniformly in $\R^n$, we can apply formula 
\eqref{eq:F_m} and the condition that the mapping $J\to\R^n$, $j\mapsto\lambda_j$, is injective,
to conclude that $\widehat{a}_j=0$ for any $j\in J_{\le R}$. Since $\bigcup_{R=1}^\infty J_{\le R}=J$ we conclude the 
statement of the lemma.
\end{proof}

\medskip

Recall from the Introduction that $\Om_{\rm p} : \R^n\to\T^M$ denotes the composed map 
$\Om_{\rm p}\equiv{\rm p}\circ\Om$ where ${\rm p} : \R^M\to\T^M$ is the standard covering map
of the torus. The following lemma gives a condition equivalent to the non-resonance condition {\rm (NC)}. 

\begin{Lem}\label{lem:density<->injectivity}
Let $\Om : \R^n\to\R^M$, $1\le n\le M$, be a linear map of maximal rank $n$.
Then, the map $\Om_{\rm p} : \R^n\to\T^M$ has a dense image in $\T^M$ if and only if
the map \eqref{eq:Lambda_m-map} is injective.
\end{Lem}

\begin{proof}[Proof of Lemma \ref{lem:density<->injectivity}]
First, let us assume that $\Om_{\rm p} : \R^n\to\T^M$ has dense image. 
Assume in addition that there exist $m',m''\in\Z^M$, $m'\ne m''$, such that $\Lambda_{m'}=\Lambda_{m''}$.
Then, for any $x\in\R^n$, $e^{i (\Lambda_{m'},x)}=e^{i (\Lambda_{m''},x)}$ or equivalently, $e^{i(m',\Om(x))}=e^{i(m'',\Om(x))}$.
This implies that
\[
\Om_{\rm p}^*(F)=0
\]
where $F(y):=e^{i(m',y)}-e^{i(m',y)}$, $y\in \T^M$. Hence, $F$ vanishes on the image of $\Om_{\rm p}$. 
Since the image of $\Om_{\rm p}$ is dense in $\T^M$, we then conclude by continuity 
that $F(y)=0$ for any $y\in\T^M$. This implies that $e^{i(m',y)}=e^{i(m',y)}$ for any $y\in\R^M$, 
that is a contradiction.

Let us now assume that the map \eqref{eq:Lambda_m-map} is injective.
Assume in addition that the image of $\Om_{\rm p}$ is not dense in $\T^M$.
Then, there exists an open ball $U$ in $\T^M$ that is not in the image of $\Om_{\rm p}$.
Take $\chi\in C^\infty(\T^M,\R)$, $\chi\ne 0$, with support in $U$, and let
$\widehat{\chi}_m$, $m\in\Z^M$, be the Fourier coefficients of $\chi$.
Clearly, the Fourier series of $\chi$ converges absolutely, and hence, independently of
the order of summation. Since, $\Om_{\rm p}^*(\chi)=0$, we obtain that
\[
\sum_{m\in\Z^M}\widehat{\chi}_m e^{i(\Lambda_m,x)}=0
\]
where the series converges absolutely. Then, we can apply e.g. Lemma \ref{lem:uniqueness}
to conclude that $\widehat{\chi}_m=0$ for any $m\in\Z^M$. By the Parseval equality we then
conclude that $\chi=0$, that contradicts our assumption that $\chi\ne 0$.
\end{proof}

Let $\Om : \R^n\to\R^M$, $1\le n\le M$, be a linear map that satisfies the non-resonance condition {\rm (NC)} 
in the Introduction and let $\Gamma_\Om$ be the discrete lattice \eqref{eq:Gamma_Om} in $\R^n$. 
We have the following lemma.

\begin{Lem}\label{lem:Gamma_Om}
For any integer $n\ge 1$ and for any $r\in\{0,...,n-1\}$ there exists an integer $M\ge n$ and a linear map $\Om : \R^n\to\R^M$ 
that satisfies the non-resonance condition {\rm (NC)} and such that $\rk\Gamma_\Om=r$.
\end{Lem}

\begin{Rem}\label{rem:quasipatterns}
Note that the quasipatterns appearing as solutions of Swift-Hohenberg PDE model and the B\'enard-Rayleigh convection
(cf. \cite{Iooss2,Iooss3} and the references therein) are quasi-periodic functions with $\Om : \R^2\to\R^4$ that 
are purely quasi-periodic (see Lemma 4 in \cite{Iooss2}).
\end{Rem}

\begin{proof}[Proof of Lemma \ref{lem:Gamma_Om}]
Since the case when $n=1$ is trivial we will assume that $n\ge 2$. 
For the simplicity of the exposition, we will first consider the case when $r=0$. 
We will construct a linear map $\Om : \R^n\to\R^{n+1}$ that satisfies {\rm(NC)} and $\Gamma_\Om=\{0\}$.
To this end, choose $\om=(\om_1,...,\om_n)\in\R^n$ such that $|\om|=1$ and the numbers 
\begin{equation}\label{eq:independence_over_Z}
\om_1,...,\om_n,1\quad\text{are linearly independent over $\Z$}.
\end{equation}
(The proof that such $\om$ exists can be done e.g. along the lines of the proof of \cite[Theorem 4, \S24]{Arnold}.)
Now, consider the $(n+1)\times n$ matrix
\begin{equation}\label{eq:Om_example}
\Om:=
\begin{pmatrix}
1&0&...&0\\
0&1&...&0\\
0&0&...&1\\
\om_1&\om_2&...&\om_n
\end{pmatrix}.
\end{equation}
The linear map $\Om : \R^n\to\R^{m+1}$ associated to this matrix satisfies {\rm (NC)} since 
\[
\Om(\om)=\big(\om_1,...,\om_n,|\om|^2\big)^T
\]
and $\big(\om_1,...,\om_n,|\om|^2\big)$ where $|\om|^2=1$ satisfy \eqref{eq:independence_over_Z}.
If we assume that $\Om(\gamma)\in\Z^{n+1}$ for some $\gamma\in\R^n$ then
we conclude from \eqref{eq:Om_example} that $\gamma=m\in\Z^n$ and
$(m,\om)\in\Z$. In view of \eqref{eq:independence_over_Z} this is possible only if
$\gamma=m=0$. This completes the proof of the case when $r=0$. 
The case when $r\ge 1$ follows by similar arguments.
\end{proof}

\medskip

For a given $f\in Q^{l,s}_\Om(\R^n)$ consider the set in $Q^{l,s}_\Om(\R^n)$,
\[
\mathcal{S}_f:=\big\{f_c(\cdot)\equiv f(\cdot+c)\,\big|\,c\in\R^n\big\}.
\]
The corollary below shows that the elements of $Q^{l,s}_\Om(\R^n)$ are {\em almost-periodic in the sense of Bochner}
with respect to the norm \eqref{eq:Q^{l,s}'-norm} in $Q^{l,s}_\Om(\R^n)$ (cf. e.g. \cite{SunTopalov,Levitan}).

\begin{Lem}\label{lem:Bochner_property}
For any $f\in Q^{l,s}_\Om(\R^n)$ the set $\mathcal{S}_f$ is precompact in $Q^{l,s}_\Om(\R^n)$.
\end{Lem}

\noindent The original Bochner's argument (\cite{Levitan}) and Lemma \ref{lem:Bochner_property} imply that any 
$f\in Q^{l,s}_\Om(\R^n)$ is {\em almost-periodic in the sense of Bohr} with respect to the norm $\|\cdot\|_{l,s}$, 
i.e. for any $\varepsilon>0$ there exists $L\equiv L_\varepsilon>0$ such that in any $n$-dimensional closed cube of side $L$ in 
$\R^n$ there exist $T>0$ (called an {\em $\varepsilon$-almost period}) such that $\|f_T-f\|_{l,s}<\varepsilon$. 
Since we will not need this statement we will omit its simple proof. The {\em classical} almost-periodic functions in the sense of Bohr
are bounded continuous functions on $\R^n$ that are almost-periodic in the sense of Bohr with respect to the norm in $L^\infty(\R^n)$
(\cite[Ch. I,\S 1]{Levitan},\cite[\S 2]{SunTopalov}).

\begin{proof}[Proof of Lemma \ref{lem:Bochner_property}]
Take $f\in Q^{l,s}_\Om(\R^n)$. Then, it follows from Lemma \ref{lem:qp_distributions1} and Lemma \ref{lem:Q^{l,s}}
that $\widehat{(f_c)}_m=e^{i(\Lambda_m,c)}\widehat{f}_m$ for any $m\in\Z^M$ where $\|f\|_{l,s}<\infty$. 
This implies that $\big|\widehat{(f_c)}_m\big|=|\widehat{f}_m|$ for any $m\in\Z^M$ and hence
\begin{equation}\label{eq:S<=T}
\mathcal{S}_f\subseteq\mathcal{T}_f\subseteq Q^{l,s}_\Om(\R^n)
\end{equation}
where $\mathcal{T}_f$ denotes the ``infinite torus'' in $Q^{l,s}_\Om(\R^n)$,
\[
\mathcal{T}_f:=\big\{g\in Q^{l,s}_\Om(\R^n)\,\big|\,|\widehat{g}_m|=|\widehat{f}_m|\,\forall m\in\Z^M \big\}.
\]
It follows easily from Cantor's diagonal argument applied to a sequence 
$(f_{c_j})_{j\ge 1}$ in $Q^{l,s}_\Om(\R^n)$ where $c_j$, $j\ge 1$, are constant vectors in $\R^n$
that $\mathcal{T}_f$ is a compact set in $Q^{l,s}_\Om(\R^n)$.
Then, $Q^{l,s}_\Om(\R^n)$ is precompact in view of the inclusion \eqref{eq:S<=T}.
\end{proof}

\end{document}